\documentclass[11pt,a4paper]{article}

\usepackage[latin1]{inputenc}
\usepackage{amsmath, amstext, amssymb, amsthm,epsfig,bbm}
\usepackage{float}

\topmargin = - 2 true cm
\newtheorem{Thm}{Theorem}
\newtheorem{Cor}[Thm]{Corollary}
\newtheorem{Lem}[Thm]{Lemma}
\newtheorem{Prop}[Thm]{Proposition}

\newtheoremstyle{remexdef}{6pt}{13pt}{}{}{\bfseries}{.}{.5em}{\thmname{#1}~\thmnumber{#2}\thmnote{(#3)}}
\theoremstyle{remexdef}

\newtheorem{Rem}[Thm]{Remark}

\numberwithin{equation}{section}

\DeclareMathOperator*{\argmax}{argmax}
\DeclareMathOperator*{\sgn}{sgn}

\begin{document}

% Nouvelles commandes

\newcommand{\N}{\mathbb{N}}
\newcommand{\Z}{\mathbb{Z}}
\newcommand{\Q}{\mathbb{Q}}
\newcommand{\R}{\mathbb{R}}
\newcommand{\C}{\mathbb{C}}
\newcommand{\LL}{\mathbb{L}}

\newcommand{\Prob}{\mathbb{P}}
\newcommand{\E}{\mathbb{E}}
\newcommand{\Filt}{\mathcal{F}}

\newcommand{\ind}{\boldsymbol{1}}

\renewcommand{\geq}{\geqslant}
\renewcommand{\leq}{\leqslant}

%%%%%%%%%%%%%%%%%%%%%%%%%%%%%%%%%%%%%%%%%%%%%%%%%%%%%%%%%%%%%%%%%%%%%%%%%%%%%%%%%%%%%%%%%%%%%%%%%%%%%%%%%%%%%%%%%%%%%%%%%%%%%%%%%%%%%%%%%%%%%%
%%%%%%%%%%%%%%%%%%%%%%%%%%%         CORPS DU DOCUMENT         %%%%%%%%%%%%%%%%%%%%%%%%%%%%%%%%%%%%%%%%%%%%%%%%%%%%%%%%%%%%%%%%%%%%%%%%%%%%%%%%
%%%%%%%%%%%%%%%%%%%%%%%%%%%%%%%%%%%%%%%%%%%%%%%%%%%%%%%%%%%%%%%%%%%%%%%%%%%%%%%%%%%%%%%%%%%%%%%%%%%%%%%%%%%%%%%%%%%%%%%%%%%%%%%%%%%%%%%%%%%%%%
\begin{center}{\bf \Large Optimal expulsion and optimal confinement\\ of a Brownian particle with a switching cost}
\vskip 16pt
{\bf Robert C.~Dalang\footnote[1]{Partially supported by the Swiss National Foundation for Scientific Research.}}  and {\bf Laura Vinckenbosch}
\vskip 12pt

\'Ecole Polytechnique F\'ed\'erale de Lausanne

%\date{\today}
%\date{}

%\maketitle
\end{center}

% abstract SPTA less than 100 words
\begin{abstract}
We solve two stochastic control problems in which a player tries to minimize or maximize the exit time from an interval of a Brownian particle, by controlling its drift. The player can change from one drift to another but is subject to a switching cost. In each problem, the value function is written as the solution of a free boundary problem involving second order ordinary differential equations, in which the unknown boundaries are found by applying the principle of smooth fit. For both problems, we compute the value function, we exhibit the optimal strategy and we prove its generic uniqueness.
\end{abstract}

%\begin{abstract}
%We solve two stochastic control problems in which a player tries to minimize or maximize the exit time from an interval of a Brownian particle, by controlling its drift. The player can change from one drift to another but is subject to a switching cost. In each problem, the value function is written as the solution of a free boundary problem involving second order ordinary differential equations, in which the unknown boundaries are found by applying the principle of smooth fit. For both problems, we compute the value function and we exhibit the optimal strategy, whose form depends on the magnitude of the switching cost. We also prove the generic uniqueness of the optimal strategy.
%\end{abstract}

{\small
%   {\bf Abbreviated title:} Expulsion and confinement of a Brownian particle
   

}

%%%%%%%%%%%%%%%%%%%%%%%%%%%%%%%%%%%%%%%%%%%%%%%%%%%%%%%%%%%%%%%%%%%%%%%%%%%%%%%%%%%%%%%%%%%%%%%%%%%%%%%%%%%%%%%%%%%%%%%%%%

\section{Description of the problem}\label{sec:problem}
Consider a game in which the player's goal is to force a Brownian particle out of an interval (say $[0,1]$) as quickly as possible. At each instant, the player selects one of two opposite constant forces, either upwards or downwards, which adds or subtracts a constant drift $\mu$ to the Brownian motion. The player is allowed to switch between the two forces at any time, but at each switch, he incurs a penalty of $c$ units of time ($c>0$). The goal is to find a strategy that minimizes the expected penalized time, that is, the sum of the time needed for the particle to exit the interval and the switching penalties (``optimal expulsion problem").

We also solve the ``opposite'' problem, in which the goal is to keep the particle inside the interval for as long as possible, subject to the same kind of switching penalty, which is now subtracted from the time to exit the interval (``optimal confinement problem").

   These two abstract problems can be viewed in the context of various applications, such as maintaining an inventory between certain bounds by controlling the production rate \cite{FlemingSethiSoner,SallesVal}, or maintaining an insurance company's capital reserve between two bounds by controlling the insurance premium \cite{bensoussan}. In certain asymptotic limits, these quantities may behave like a Brownian motion, and a change of production rate or of premium may entail a switching cost. The two boundaries may represent certain levels that one may want to reach as soon (or as late) as possible. 
   
   The presence of the switching cost is the key issue here: for instance, Prokhorov \cite{Prohorov64} solves a similar problem but without cost penalty, and Mandl~\cite{Mandl67} treats a control problem for a Brownian motion under a constraint on the number of switchings. When there is no switching cost, then the solutions of these problems are well-known (see \cite[p.167-168]{FlemingSoner06}).
   
   There is a well-developed literature for studying this kind of stochastic control problem, including \cite{Fleming75, FlemingSoner06, Krylov77, OksendalSulemBook, YongZhou99}. Most frequently, these problems involve terminal costs and running costs. More recently, even more general kinds of performance criteria have been considered, as in ~\cite{Peskir2005}, where the criterion also involves the running maximum of the observed process.
%some problems which performance criterion depends on the maximum of sample path have been explored as in~\cite{Peskir2005}. 
In the presence of a switching cost, the problem falls into the theory of impulse control, as described for instance in \cite[Chapter 6]{OksendalSulemBook}. 
   
   In order to solve our two control problems,
%\eqref{eq:control:problem:min} or \eqref{eq:control:problem:max}, 
we begin by formulating a free boundary problem for the value function. This involves splitting the state space into two regions, a {\em continuation} region and a {\em switching} region.  The particular form of the regions is guessed from the description of the problem. In the continuation region, the value function solves certain ordinary differential equations, and in the complement of this region, the value function satisfies a relationship related to the switching cost. There are also boundary conditions at the extremities of the interval. In general, this system of equations is {\em not} sufficient to characterize the value function, and this is indeed the case here: it is necessary to specify appropriate additional conditions at the free boundaries between the regions, which we do using the so-called \emph{principle of smooth fit} (see for example \cite[p.147]{Peskir06} and \cite[Section 5.3.4]{Pham}). 

   This approach has a long history, going back to \cite{Chernoff61,Grigelionis_Shiryaev66,shepp69,shiryaev}, and, more recently, \cite{Peskir06}, and has proved to be quite successful in a wide range of problems, including, in addition to those in the references just mentioned, the problem of optimal switching (without cost) between two Brownian motions \cite{Mandelbaum_Shepp_Vanderbei90}. Other examples of optimal switching problems related to ours and with explicit solutions can be found in ~\cite{Bayraktar2010, Duckworth2001, LyVath2007}. The problems considered in these papers differ from ours in particular because the state space is either the real line or the half-line, and there are no boundary conditions.
	
%	See~\cite{Bayraktar2010, Duckworth2001, LyVath2007} for other examples of optimal switching problems for which an explicit solution has been exhibited. In these latter works, the general context is very close to the present problem except for the state space that is either the real line or the half-line, whereas we consider a bounded interval in our case.

   In our problem, if the switching cost is high enough, then, obviously, one should switch drifts rarely or not at all, and in fact, it turns out that there is a critical value $c^*(\mu)$, which turns out to be the same in both the expulsion and confinement problems, above which it is optimal never to switch drifts. We compute this value explicitly, and then we show that for costs $c < c^*(\mu)$, the optimal strategy is determined by four thresholds that are the endpoints of the switching regions. We determine these thresholds explicitly, up to the resolution of a single transcendental equation (in each problem).
% (see also \cite{Fleming75,FlemingSoner06,Krylov77,YongZhou99}).
%These references use the so-called \emph{principle of smooth fit} (see for example~\cite[p.147]{Peskir06} and \cite[Section 5.3.4]{Pham}) to determine a unique solution to the free boundary problem. This method has proven to be quite successful, including for instance in the problem of optimal switching (without cost) between two Brownian motions \cite{Mandelbaum_Shepp_Vanderbei90}. We also use this method here.

   This paper is organized as follows. In Section \ref{sec:solution}, we give the precise formulation of the optimal expulsion and confinement problems, we guess the form of the continuation and switching regions in order to state the free boundary problems and their solutions, we identify the critical cost $c^*(\mu)$, and we present our main results concerning the value functions and the description of the optimal controls. 
%It turns out that the structure of the optimal solution depends on whether or not $c$ is smaller than a critical value $c^*(\mu)$, which is given explicitly in \eqref{eq:def:c*}. 
In Section \ref{sec:proofs}, we solve the free boundary problems, by considering first the case $c = c^*(\mu)$, and then use this for the cases $c< c^*(\mu)$ and  $c> c^*(\mu)$. Using the {\em verification method,} we prove that the solutions are indeed the value functions, by identifying a process that plays the role of Snell's envelope \cite{elkaroui}: it is a sub- (or super-) martingale for all strategies and a martingale for the optimal strategy. This proof makes use of the local time-space formula of \cite{Peskir07}. Finally, in Section \ref{sec:further}, we identify the consequences of a suboptimal action, which allows us to show that the optimal solutions are generically unique (except in the critical case $c= c^*(\mu)$, where there are two distinct optimal solutions), and we study the limiting case $c \downarrow 0$, from which we recover the zero-cost case of \cite{FlemingSoner06}.

\section{Formulation of the problem and main results}\label{sec:solution}

Let $(B_t)_{t\geq0}$  be a standard Brownian motion, defined on a probability space $(\Omega, \Filt, \Prob)$, such that $B_0=0$ a.s., let $(\Filt_t)_{t\geq 0}$ be its natural filtration and let $\mathcal{A}$ denote the set of all $\Filt_t$-adapted processes that are right continuous, piecewise constant and take values in $\{-1,1\}$. The elements of $\mathcal{A}$ are the \emph{strategies} available to the player. We consider a control model in which the system's state is given by the stochastic differential equation
\begin{equation}\label{eq:SDE:state}
dX^A_t =A_t\,\mu\,dt+dB_t,
\end{equation}
where $A=(A_t)_{t\geq0}\in\mathcal{A}$ and $\mu>0$ is a given positive constant. The random variable $X^A_t$ denotes the \emph{position of the particle at time} $t$ if the player is using the strategy $A$, and $A_t$ gives the direction in which the player is pushing at time $t$. The initial conditions are given by a family of probability measures $\left\{\Prob_{x,a},\,x\in[0,1],\,a\in\{\pm1\}\right\}$ defined by $\Prob_{x,a}(X^A_0=x,\, A_{0-}=a)=1$, with associated expectations $\E_{x,a}$. Here, $A_{0-}$ is the drift that applies just before time $0$, and which can change at time $0$ precisely, if desirable. Let $c>0$  be the \emph{switching cost} and let
\begin{equation}\label{eq:def:NtA}
N_t(A)=\sharp\{s\in\,[0,t]: A_{s-}\neq A_{s}\}
\end{equation}
be the number of switches of drift of the process  $X^A$ up to time $t$. Notice that $N_0(A)>0$ is possible. The \emph{cost function} for the minimization (resp. maximization) problem is then given by
$$J_c(x,a,A)=\E_{x,a}(\tau^A + c N_{\tau^A}(A)),$$
respectively,
$$J^{\text{max}}_c(x,a,A)=\E_{x,a}(\tau^A - c N_{\tau^A}(A)),$$
where
\begin{equation}\label{eq:def:tauA}
\tau^A=\inf\{t\geq0: X^A_t\notin \,]0,1[\}
\end{equation}
and the \emph{value functions} are respectively given by
\begin{align}
V_c(x,a)&=\inf_{A\in\mathcal{A}} J_c(x,a,A),\label{eq:control:problem:min}\\
V^{\text{max}}_c(x,a)&=\sup_{A\in\mathcal{A}} J^{\text{max}}_c(x,a,A).\label{eq:control:problem:max}
\end{align}
The goal is then to compute these value functions and to find \emph{optimal controls} $A^*\in\mathcal{A}$ and $G^*\in\mathcal{A}$ such that $V_c(x,a)=J_c(x,a,A^*)$ and $V^{\text{max}}_c(x,a)=J^{\text{max}}_c(x,a,G^*)$, for all $c>0$ and for all $\mu>0$.

\subsection{Properties of the solution}\label{sec:properties}
%In this subsection, we present the solution of both problems \eqref{eq:control:problem:min} and \eqref{eq:control:problem:max}. 
In order to formulate a free boundary problem for each value function, we will assume that the solution will satisfy three properties. The validity of these properties will be established in Section~\ref{sec:proofs}.\\

\noindent
\textbf{Property 1.} \emph{The optimal strategy is symmetric with respect to the initial drift and the value functions satisfy}
$V_c(x,a)=V_c(1-x, -a)$  \emph{and} $V^{\text{max}}_c(x,a)=V^{\text{max}}_c(1-x, -a)$,
\emph{for all} $x\in[0,1]$ \emph{and} $a\in\{\pm1\}$.\\

\noindent
\textbf{Property 2.} \emph{There exists a critical value $c^*(\mu)>0$ for which the optimal strategy is the constant strategy if $c>c^*(\mu)$.}\\

Indeed, for a given $\mu$, the expected exit time of a Brownian motion with constant drift $\pm\mu$ is a bounded function of the starting point $x\in[0,1]$. Thus, if the cost exceeds a certain value, then a reasonable player will never pay this cost to change the initial drift. This value is given by the maximal difference between the expected exit time from  $[0,1]$ of a Brownian motion starting at $x$ with a constant drift $\mu$ or $-\mu$. Namely, for $\nu\in\R$, set $\sigma^\nu=\inf\{t\geq 0: B_t+\nu t\notin\,]0,1[\,\}$ and $f^\nu(x)=\E_x(\sigma^\nu)$. Then, by taking the derivative at $0$ of the moment-generating function of a Brownian motion with drift (see~\cite[II.2.3]{Borodin_salminen02}), we find, after tedious calculations, that
\begin{equation}\label{eq:def:fnu}
f^\nu(x)=-\frac{x}{\nu}+\frac{1-e^{-2\nu x}}{\nu(1-e^{-2\nu})}
\end{equation}
and that
\begin{align}
c^*(\mu)&:=\max_{x\in[0,1]}\left(f^\mu(x)-f^{-\mu}(x)\right)=\max_{x\in[0,1]}\left( -\frac{2x}{\mu}+\frac{1}{\mu}+\frac{e^{2\mu x}-e^{2\mu(1-x)}}{\mu(e^{2\mu}-1)}\right)\nonumber\\
&=\frac{-1}{\mu^2}\left\{\log\left(\tfrac{\sinh(\mu)}{\mu}\left(1-\sqrt{1-\tfrac{\mu^2}{\sinh^2(\mu)}}\right) \right)+\sqrt{1-\tfrac{\mu^2}{\sinh^2(\mu)}}\right\}\,,\label{eq:def:c*}
\end{align}
for any $\mu>0$ (note that $\lim_{\mu\downarrow 0}c^*(\mu)=0$ and $\lim_{\mu\rightarrow +\infty}c^*(\mu)=0$). This maximum is attained at
\begin{equation}\label{eq:maximum:diff:fnu}
x^*=\frac{1}{2\mu}\left(\log\left(\tfrac{\sinh(\mu)}{\mu}\left(1-\sqrt{1-\tfrac{\mu^2}{\sinh^2(\mu)}}\right)\right)+\mu\right).
\end{equation}

The third property concerns the general shape of the optimal strategy. Indeed, consider two scenarios in the minimization problem. Assume first that one starts near $1$ with a positive drift. The player will keep this favorable drift for a while. If the particle goes down, then when it reaches $\frac{1}{2}$, both drifts are equivalent because of the symmetry property. Since the player is subject to a switching penalty, he will keep the positive drift. If the particle  keeps going down, then it will become more advantageous to change to a negative drift so that the particle will exit more quickly through $0$.

Secondly, if one starts close to $0$ with a positive drift, then the diffusive behavior of the particle makes it very likely that it will rapidly hit $0$ even if the drift is in the unfavorable direction. Thus, it is probably not worthwhile to pay the penalty to change the drift. These two facts are summarized by the following property:\\

\noindent
\textbf{Property $3$.} \emph{There exist two barriers $a_c$ and $b_c$ satisfying $0<a_c\leq b_c<\frac{1}{2}$ and such that it is optimal to keep a positive drift above $b_c$ or below $a_c$ and it is optimal to switch to a negative drift within $[a_c,b_c]$.}\\

In the case of the maximization problem, this property becomes:\\

\noindent
\textbf{Property $3^{\text{max}}$.} \emph{There exist two barriers $a^{\text{max}}_c$ and $b^{\text{max}}_c$ satisfying $\frac{1}{2}<a^{\text{max}}_c\leq b^{\text{max}}_c<1$ and such that it is optimal to keep a positive drift above $b^{\text{max}}_c$ or below $a^{\text{max}}_c$ and it is optimal to switch to a negative drift within $[a^{\text{max}}_c,b^{\text{max}}_c]$.}

\subsection{Solution of the minimization problem}
Via the \emph{the dynamic programming principle} and the Hamilton-Jacobi-Bellman equation (see e.g.~\cite[Section 3.1]{OksendalSulemBook}, \cite[Section 3]{YongZhou99}, or \cite[Chapter 5]{Pham}), the value function $V_{c}(x,\pm1)$ of the stochastic control problem~(\ref{eq:control:problem:min}) should satisfy the following system of variational inequalities:
\begin{align*}
\min\left\{ 1\pm \mu \frac{\partial V_c}{\partial x}(x,\pm1)+\frac{\partial^{2} V_c}{\partial x^{2}}(x,\pm1),\, -V_{c}(x,\pm 1)+V_{c}(x,\mp1)+c\right\}&=0, \; x\,\in\, ]0,1[\,,\\
V_{c}(0,\pm1)=V_{c}(1,\pm 1)&=0.
\end{align*}
We formulate this system as a \emph{free boundary problem} for the value function. In the region where it is optimal to keep the current drift, the value function must satisfy the ordinary differential equation~(\ref{eq:ODE:continuation}) below. In the region where it is optimal to switch to the other drift, we have the equation~(\ref{eq:switching}). With the three properties of Section \ref{sec:properties}, we expect that the value function should satisfy the following problem, in which $V_c(x+,a)$ (resp. $V_c(x-,a)$) denotes $\lim_{y\downarrow x}V_c(y,a)$ (resp. $\lim_{y\uparrow x}V_c(y,a)$):
\begin{subequations}\label{eq:FBP:Vc}
\begin{align}
\mu \frac{\partial V_c}{\partial x} (x,1)&+ \frac{1}{2}\,\frac{\partial^2 V_c}{\partial x^2} (x,1)=-1, && x\in[0,a_c[\,\cup\,]b_c,1]&& \label{eq:ODE:continuation}\\
V_c(x,1)&= V_c(x,-1)+ c, && x\in[a_c, b_c]&& \label{eq:switching} \\
V_c(0,1)&= V_c(1,1)=0, &&&& \mbox{(boundary conditions)}\label{eq:boundary:cond} \\
V_c(a_c-,1)&= V_c(a_c+,1),&&&& \mbox{(continuous fit)} \label{eq:fit1} \\
V_c(b_c-,1)&= V_c(b_c+,1), &&&& \mbox{(continuous fit)} \label{eq:fit2}\\
\frac{\partial V_c}{\partial x}(a_c-,1)&= \frac{\partial V_c}{\partial x}(a_c+,1),&&&& \mbox{(smooth fit)}\label{eq:smooth:fit1}\\
\frac{\partial V_c}{\partial x}(b_c-,1)&= \frac{\partial V_c}{\partial x}(b_c+,1), &&&& \mbox{(smooth fit)}\label{eq:smooth:fit2}\\
V_c(x,-1)&=V_c(1-x,1),  && x\in[0,1]  && \mbox{(symmetry)}\label{eq:symmetry}
\end{align}
\end{subequations}
where $a_c$ and $b_c$ are two unknowns satisfying $0<a_c\leq b_c<\frac{1}{2}$.

%\begin{Rem}
%The compact form of  the dynamic programming equation for the value function $V_{c}$ is 
%\end{Rem}

\begin{Prop}\label{prop:sol:FBP:min}
Let $c^*(\mu)$ be given by~(\ref{eq:def:c*}). There exists a unique solution $\{\bar{V}_c,a_c,b_c\}$ to the free boundary problem~(\ref{eq:FBP:Vc}).
\begin{enumerate}
  \item If $c=c^*(\mu)$, then the solution is given by
    \begin{equation}\label{eq:sol:FBP:c*:ac}
    a_{c^*}=b_{c^*}=\frac{1}{2\mu}\left(\log\left(\tfrac{\sinh(\mu)}{\mu}\left(1-\sqrt{1-\tfrac{\mu^2}{\sinh^2(\mu)}}\right)\right)+\mu\right)
    \end{equation}
    and
    \begin{equation}\label{eq:sol:FBP:c*:V}
    \bar{V}_{c^*}(x,a)=f^{a\mu}(x)=J_{c^*}(x,a,\tilde{A}), \quad\qquad\qquad\; x\in[0,1],\,a\in\{\pm1\},
    \end{equation}
    where $\tilde{A} = (\tilde{A}_t \equiv a)_{t\geq 0}$ is the constant strategy and $f^{a\mu}$ is defined by~(\ref{eq:def:fnu}). Moreover,
   \begin{equation}\label{eq:ac*:argmax}
     a_{c^*}=\argmax_{x\in[0,1]}\left\{f^{\mu}(x)-f^{-\mu}(x)\right\}\quad\text{ and }\quad f^\mu(a_{c^*})=f^{-\mu}(a_{c^*})+c^*(\mu).
   \end{equation}

  \item If $\,0<c<c^*(\mu)$, then the solution is given by
        \begin{align}\label{eq:sol:FBP:min:c<c*}
        \bar{V}_c(x,1)&=\left\{
           \begin{array}{ll}
            \displaystyle   -\tfrac{x}{\mu}+\beta_c \left(e^{-2\mu x}-1\right), & x\in[0,a_c[\,, \\
            \displaystyle   \tfrac{x}{\mu}+\alpha_c \left(e^{2\mu x}-1\right)+ c, & x\in[a_c,b_c], \\
            \displaystyle   \tfrac{1-x}{\mu}+\alpha_c \left(e^{2\mu(1-x)}-1\right), & x\in\,]b_c,1], \\[1ex]
           \end{array}
         \right.\\
         \bar{V}_c(x,-1)&=\,\bar{V}_c(1-x,1),\hspace{18ex} x\in[0,1],\nonumber
        \end{align}
        where $b_c\in\left[0,\frac{1}{2}\right[$ is the unique solution $x$ of the transcendental equation
        \begin{equation}\label{eq:bc}
        e^{4\mu x-2\mu} ( 2\mu x-\mu+c\mu^2-1) + 2\mu x-\mu+c\mu^2+1=0,
        \end{equation}
        \begin{equation}\label{eq:alpha}
        \alpha_c=\frac{-e^{-\mu}}{2\mu^2\cosh(2\mu b_c-\mu)},
        \end{equation}
        $a_c$ is the unique solution $y\in\,]0,b_c[$ of the transcendental equation
        \begin{equation}\label{eq:ac}
        \mu^2\alpha_c e^{4\mu y}+ (1-2\mu^2\alpha_c)e^{2\mu y} -2\mu y +\mu^2\alpha_c  -c\mu^2-1=0
        \end{equation}
        and
        \begin{equation}\label{eq:beta}
        \beta_c=-\alpha_c e^{4\mu a_c}-\frac{1}{\mu^2}\,e^{2\mu a_c}.
        \end{equation}
\end{enumerate}
\end{Prop}
%
%
%
%
%construction of the strategy $A^c$
We now use the value of the barriers $a_c$ and $b_c$ to define the following four subsets of $[0,1]$:
\begin{align*}
C_1&= [0,a_c[\,\cup\,]b_c,1],       &&C_{-1}= [0,1-b_c[\,\cup\,]1-a_c,1],\\
D_1&= [a_c,b_c],                    &&D_{-1}= [1-b_c,1-a_c].
\end{align*}
The subsets $C_a$ and $D_a$ are called respectively the \emph{continuation} and the \emph{switching} region for the drift $a$. For $0<c\leq c^*$, we define the candidate optimal strategy $A^c$ as follows. Let $(x,a)\in[0,1]\times\{-1,1\}$ be the initial conditions, define inductively an increasing sequence $(\tau_n)_{n\in\N}$ of stopping times by $\tau_0=0$ and for $n\geq0$,
$$\tau_{n+1}=\left\{\begin{array}{ll}
\inf\left\{t\geq\tau_n: X_t^n\in D_{(-1)^n a}\right\}, &\mbox{ if } \{\cdots\}\neq\emptyset, \\
+\infty, &\mbox{ otherwise},
\end{array}\right.$$
where $X_{0}^{-1}=x$,  and for $n\geq0$,  $X^n$ is the process defined as the solution of
\begin{align*}
dX^n_t&=(-1)^n a\mu\,dt+dB_t, \qquad t\in[\tau_n,+\infty[\,, \\
X^n_{\tau_n}&=X^{n-1}_{\tau_n}.
\end{align*}
Set $\tau=\inf\left\{t\geq0: \exists\, n\geq0 \text{ with } X_t^n\notin\,]0,1[ \text{ and } t\in[\tau_n,\tau_{n+1}[\,\right\}$. Then define $A^c_{0-}=a$, $A^c_t=a$ for $t\in[\tau_0,\tau_1[\,$, and for $n\geq1$,
\begin{equation}\label{eq:def:candidate:Ac:a}
A^c_t=-A^c_{\tau_n-},   \qquad\text{ for } t\in[\tau_n,\tau_{n+1}[\,.
\end{equation}

This construction implies that $A^c$ satisfies for all $t\geq0$
\begin{equation}\label{eq:def:candidate:Ac}
A^c_t=
\left\{
  \begin{array}{ll}
    A^c_{t-},   &\qquad \text{ if } X^{A^c}_t \in C_{A^c_{t-}}, \\
    -A^c_{t-},  &\qquad \text{ if } X^{A^c}_t \in D_{A^c_{t-}}, \\
  \end{array}
\right.
\end{equation}
and the controlled process $X^{A^c}$ is the solution of $dX^{A^c}_t=A^c_t\mu\,dt+dB_t$ and $X^{A^c}_0=x$. Observe that the sequence of the stopping times $(\tau_n)$ corresponds to the jump times of the strategy $A^c$, that $X^{A^c}_t=X^n_t$ on $[\tau_n,\tau_{n+1}[$ and that $\tau=\tau^{A^c}$, the exit time from  $]0,1[$ of $X^{A^c}$. This candidate strategy is pictured in Figure~\ref{figure:controle:AcGc} and it satisfies the following properties.

\begin{Prop}\label{prop:esperance:tauAc}
If $0<c\leq c^*(\mu)$, then for all $x\in[0,1]$ and $a\in\{-1,1\}$, we have
$$\E_{x,a}\left(\tau^{A^c}\right)<+\infty\quad\mbox{ and } \quad \E_{x,a}\left(N_{\tau^{A^c}}(A^c)\right)<+\infty.$$
\end{Prop}

%Optimality thm's
Now that we have exhibited a candidate strategy and a candidate value function as the solution of a free boundary problem, we can state the optimality theorem.

\begin{Thm}\label{verification:thm:min}
Let $c^*(\mu)$ be given by~(\ref{eq:def:c*}).  If $A^c$ is the control satisfying~(\ref{eq:def:candidate:Ac:a}) (and~(\ref{eq:def:candidate:Ac})) and if $\bar{V}_c$ denotes the unique solution of the free boundary problem~(\ref{eq:FBP:Vc}) given in Proposition~\ref{prop:sol:FBP:min}, then:
\begin{enumerate}
  \item for $0<c\leq c^*(\mu)$, the value function is $V_c = \bar{V}_c$ and $A^c$ is an optimal control for the problem~(\ref{eq:control:problem:min});

  \item for $c>c^*(\mu)$, the value function is $V_c = \bar{V}_{c^*(\mu)}$ and $\tilde{A}=(\tilde{A}_t\equiv a)_{t\geq0}$ is $\Prob_{x,a}$-a.s.~the unique optimal control for the problem~(\ref{eq:control:problem:min}).
\end{enumerate}
\end{Thm}

\begin{Rem}
In the case where $c=c^*(\mu)$, we see in Proposition~\ref{prop:sol:FBP:min} that the switching regions consist of one single point, at which a change of drift can be considered as insignificant. Indeed at this point, the price to pay for a change of drift is equal to the maximum expected profit provided by this change itself. Moreover, the same Proposition together with Theorem~\ref{verification:thm:min} show that the constant strategy and the candidate $A^{c^*}$ are both optimal in this case. The same remark will apply to the maximization problem. For a uniqueness result when $0<c<c^*(\mu)$, see Proposition \ref{prop:uniqueness}.
\end{Rem}

\subsection{Solution of the maximization problem}
Similar to the minimization problem, the value function $V^{\text{max}}_c(x,\pm1)$ should satisfy
\begin{align*}
\max\left\{ 1\pm \mu \frac{\partial V^{\text{max}}_c}{\partial x}(x,\pm1)+\frac{\partial^{2} V^{\text{max}}_c}{\partial x^{2}}(x,\pm1),\, -V^{\text{max}}_{c}(x,\pm 1)+V^{\text{max}}_{c}(x,\mp1)-c\right\}&=0, \; x\,\in\, ]0,1[\,,\\
V^{\text{max}}_{c}(0,\pm1)=V^{\text{max}}_{c}(1,\pm 1)&=0.
\end{align*}
In view of the three properties of Section \ref{sec:properties}, we formulate this system as the following free boundary problem for the value function:% and by the dynamic programming principle:
\begin{subequations}\label{eq:FBP:Wc}
\begin{align}
&\mu \frac{\partial V^{\text{max}}_c}{\partial x}(x,1)+ \frac{1}{2}\frac{\partial^2 V^{\text{max}}_c}{\partial x^2}(x,1)=-1, && x\in[0,a^{\text{max}}_c[\,\cup&&\hspace{-4ex}]b^{\text{max}}_c,1] \label{eq:continuation:max}\\
&V^{\text{max}}_c(x,1)= V^{\text{max}}_c(x,-1)-c,                       &&x\in[a^{\text{max}}_c, b^{\text{max}}_c]&&             \label{eq:switching:max} \\
&V^{\text{max}}_c(0,1)= V^{\text{max}}_c(1,1)=0,                        &&&&\mbox{(boundary conditions)}\label{eq:boundary:condit:max} \\
&V^{\text{max}}_c(a^{\text{max}}_c-,1)= V^{\text{max}}_c(a^{\text{max}}_c+,1),                      &&&& \mbox{(continuous fit)} \label{eq:cond:fit1:max} \\
&V^{\text{max}}_c(b^{\text{max}}_c-,1)= V^{\text{max}}_c(b^{\text{max}}_c+,1),                      &&&& \mbox{(continuous fit)} \label{eq:cond:fit2:max}\\
&\frac{\partial V^{\text{max}}_c}{\partial x}(a^{\text{max}}_c-,1)= \frac{\partial V^{\text{max}}_c}{\partial x}(a^{\text{max}}_c+,1), &&&& \mbox{(smooth fit)}\label{eq:smooth:fit1:max}\\
&\frac{\partial V^{\text{max}}_c}{\partial x}(b^{\text{max}}_c-,1)= \frac{\partial V^{\text{max}}_c}{\partial x}(b^{\text{max}}_c+,1), &&&& \mbox{(smooth fit)}\label{eq:smooth:fit2:max}\\
&V^{\text{max}}_c(x,-1)=V^{\text{max}}_c(1-x,1),             && x\in[0,1]   &&\mbox{(symmetry)}\label{eq:symmetry:max}
\end{align}
\end{subequations}
where $a^{\text{max}}_c$ and $b^{\text{max}}_c$ are two unknowns satisfying $\frac{1}{2}< a^{\text{max}}_c\leq b^{\text{max}}_c<1$. Even though the solution of the free boundary problem \eqref{eq:FBP:Wc} is similar to the one in the minimization problem, the two problems are not symmetric as we shall see in the next proposition.

%\begin{Rem}
%As for the minimization problem, we can write the dynamic programming equation for the value function $V^{\text{max}}_{c}$ in its compact form: 
%\end{Rem}

\begin{Prop}\label{prop:sol:FBP:max}
Let $c^*(\mu)$ be given by~(\ref{eq:def:c*}). There exists a unique solution $\{\bar{V}^{\text{max}}_c,a^{\text{max}}_c,b^{\text{max}}_c\}$ to the free boundary problem~(\ref{eq:FBP:Wc}).
\begin{enumerate}
\item If $c=c^*(\mu)$, then the solution is given by $a^{\text{max}}_c=b^{\text{max}}_c=1-a_c$ and
\begin{equation}\label{eq:sol:FBP:c*:W}
\bar{V}^{\text{max}}_{c^*}(x,a)=\bar{V}_{c^*}(x,a)=f^{a\mu}(x), \quad\qquad\qquad\; x\in[0,1],\,a\in\{\pm1\},
\end{equation}
where $a_c$ and $\bar{V}_{c^*}$ are given by Proposition~\ref{prop:sol:FBP:min} and $f^{a\mu}$ is defined by~(\ref{eq:def:fnu}).

\item If $\,0<c<c^*(\mu)$, then the solution is given by
\begin{align}\label{eq:sol:FBP:max:c<c*}
\bar{V}^{\text{max}}_c(x,1)&=\left\{
           \begin{array}{ll}
            \displaystyle   -\tfrac{x}{\mu}+\gamma_c \left(e^{-2\mu x}-1\right), & x\in[0,a^{\text{max}}_c[\,, \\
            \displaystyle   -\tfrac{1-x}{\mu}+\gamma_c \left(e^{-2\mu(1-x)}-1\right)-c, & x\in[a^{\text{max}}_c,b^{\text{max}}_c], \\
            \displaystyle   \tfrac{1-x}{\mu}+\delta_c \left(e^{2\mu(1-x)}-1\right), & x\in\,]b^{\text{max}}_c,1], \\
           \end{array}
         \right.\\
         \bar{V}^{\text{max}}_c(x,-1)&=\,\bar{V}^{\text{max}}_c(1-x,1),\hspace{21ex} x\in[0,1],\nonumber
        \end{align}
where $a^{\text{max}}_c=1-b_c$ and $\gamma_c=\alpha_c e^{2\mu}$ with $b_c$ and $\alpha_c$ given by Proposition~\ref{prop:sol:FBP:min},
$b^{\text{max}}_c$ is the unique solution $x\in\,]a^{\text{max}}_c,1[$ of the transcendental equation
\begin{equation}\label{eq:g2}
\gamma_c\mu^2e^{4\mu x-4\mu} +e^{2\mu x-2\mu}\left(1-2\gamma_c\mu^2\right)- 2\mu x + 2\mu-1+\gamma_c\mu^2+c\mu^2=0
\end{equation}
and
\begin{equation}\label{eq:delta}
\delta_c=-\tfrac{1}{\mu^2}\,e^{-2\mu(1-b^{\text{max}}_c)}-\gamma_c e^{-4\mu(1-b^{\text{max}}_c)}.
\end{equation}
Moreover, $b^{\text{max}}_c>1-a_c$.
\end{enumerate}
\end{Prop}

We construct the candidate strategy for the maximization problem as we did to get to~(\ref{eq:def:candidate:Ac}). Let
\begin{align*}
C_1^{\text{max}}&= [0,a^{\text{max}}_c[\,\cup\,]b^{\text{max}}_c,1], &&C_{-1}^{\text{max}}= [0,1-b^{\text{max}}_c[\,\cup\,]1-a^{\text{max}}_c,1],\\
D_1^{\text{max}}&= [a^{\text{max}}_c,b^{\text{max}}_c],              &&D_{-1}^{\text{max}}= [1-b^{\text{max}}_c,1-a^{\text{max}}_c],
\end{align*}
with $a^{\text{max}}_c$ and $b^{\text{max}}_c$ given in Proposition~\ref{prop:sol:FBP:max}. We denote by $G^c$ the strategy constructed using the ideas that led to \eqref{eq:def:candidate:Ac:a} and that satisfies
\begin{equation}\label{eq:def:candidate:Gc}
G^c_t=
\left\{
  \begin{array}{ll}
    G^c_{t-}, & \mbox{ if } X^{G^c}_t \in C^{\text{max}}_{G^c_{t-}}, \\
    -G^c_{t-}, & \mbox{ if } X^{G^c}_t \in D^{\text{max}}_{G^c_{t-}}, \\
  \end{array}
\right.
\end{equation}
where $X^{G^{c}}$ is the process controlled by $G^{c}$. This strategy is pictured in Figure~\ref{figure:controle:AcGc} and it satisfies the following properties.

\begin{figure}[h]
\begin{center}
\includegraphics[width=7cm]{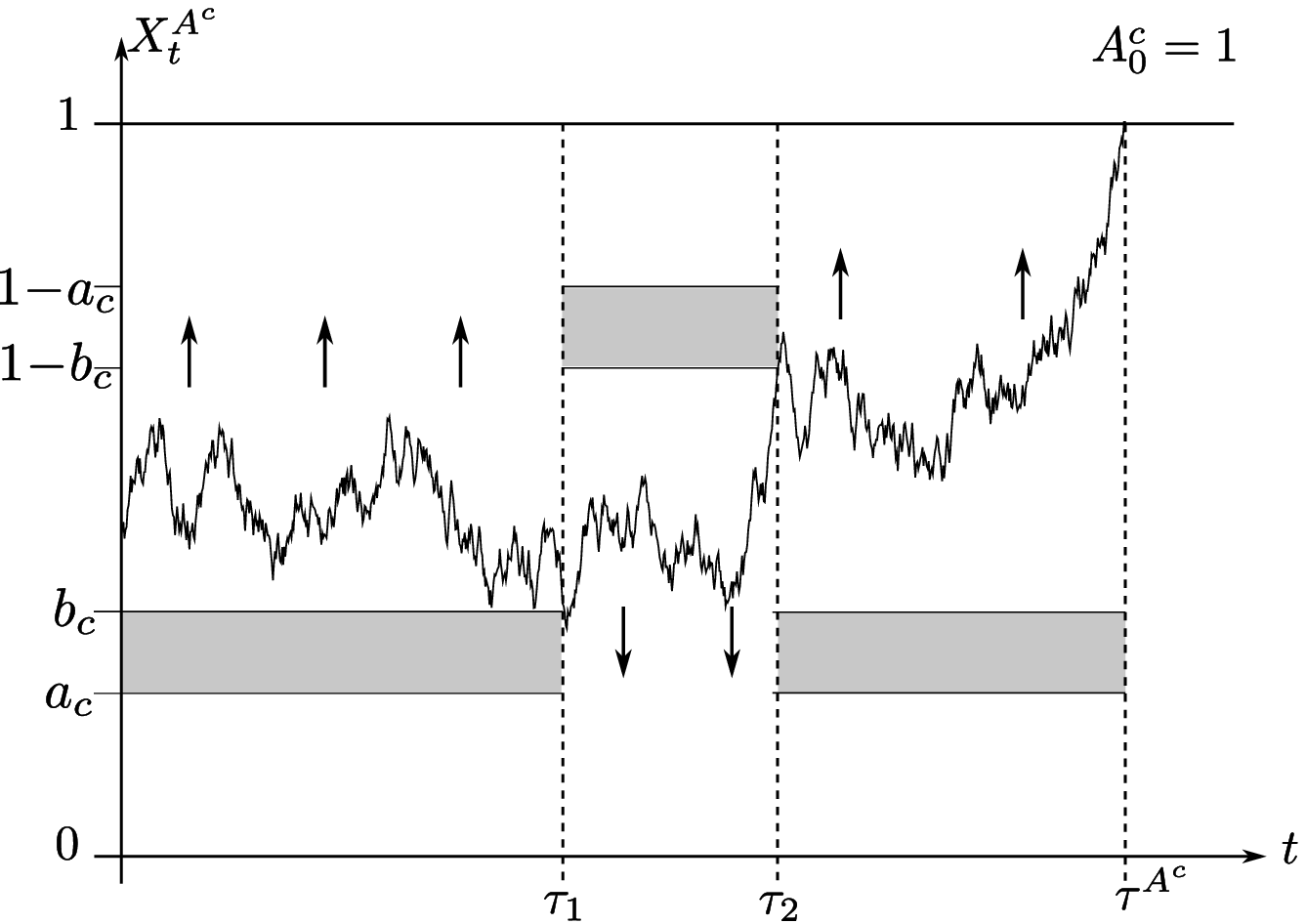}\hspace{5ex}\includegraphics[width=7cm]{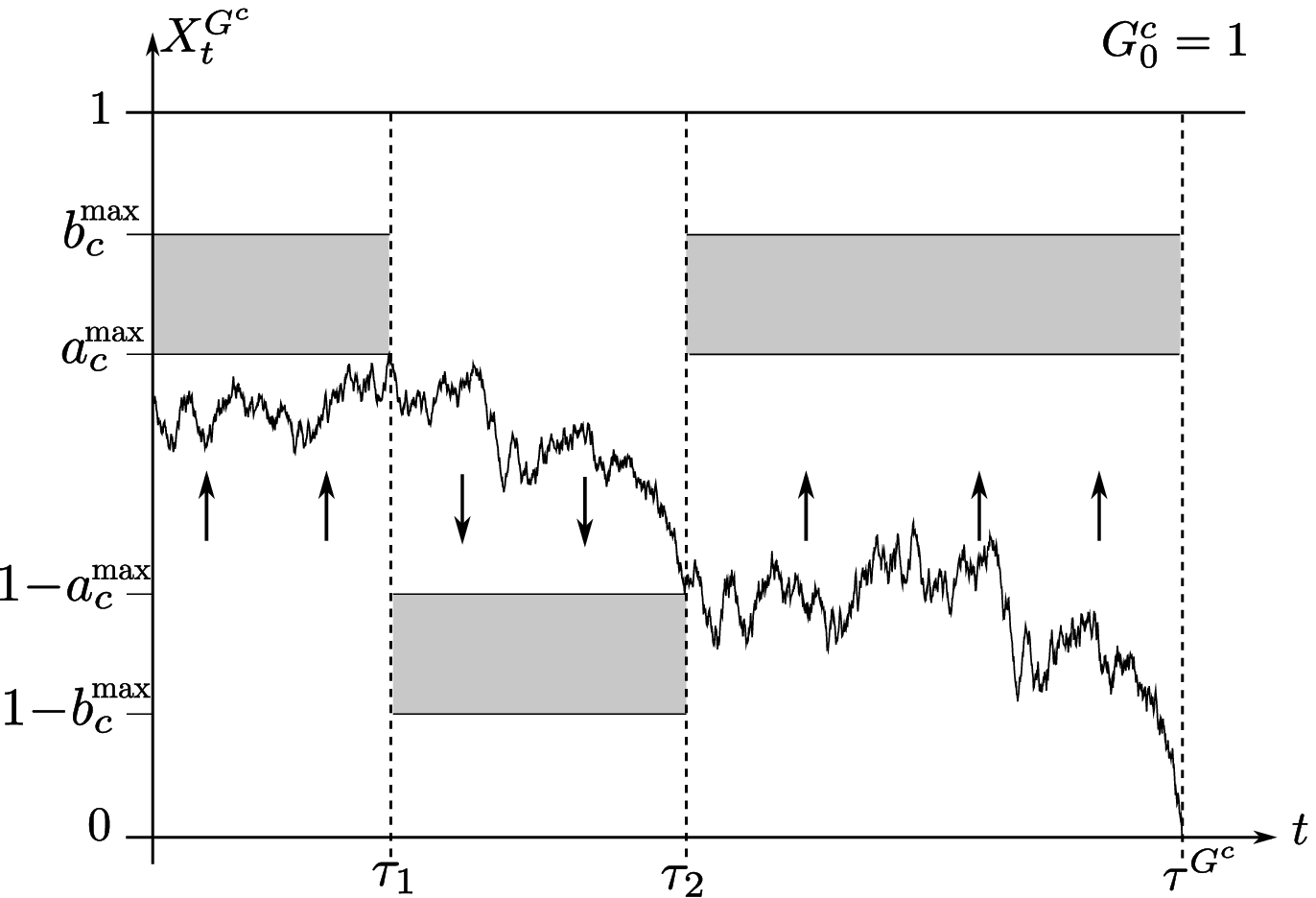}
\end{center}
\caption{Illustration of the control $A^c$ on the left-hand side and of $G^c$ on the right-hand side in the case where $0<c<c^*(\mu)$. Here we assume that the initial drift is positive.}\label{figure:controle:AcGc}
\end{figure}

\begin{Prop}\label{prop_esperance_tauGc}
If $0<c\leq c^*(\mu)$, then for all $x\in[0,1]$ and $a\in\{-1,1\}$, we have
$$\E_{x,a}\left(\tau^{G^c}\right)<+\infty\quad\mbox{ and } \quad \E_{x,a}\left(N_{\tau^{G^c}}(G^c)\right)<+\infty.$$
\end{Prop}

\vskip 12pt

\begin{Thm}\label{verification:thm:max}
Let $c^*(\mu)$ be given by~(\ref{eq:def:c*}).  If $G^c$ is the control satisfying~(\ref{eq:def:candidate:Gc}) and if $\bar{V}^{\text{max}}_c$ denotes the unique solution of the free boundary problem~(\ref{eq:FBP:Wc}) given in Proposition~\ref{prop:sol:FBP:max}, then:
\begin{enumerate}
  \item for $0<c\leq c^*(\mu)$, the value function is $V^{\text{max}}_c = \bar{V}^{\text{max}}_c$ and $G^c$  is an optimal control for the problem~(\ref{eq:control:problem:max});

  \item for $c>c^*(\mu)$, the value function is $V^{\text{max}}_c = \bar{V}^{\text{max}}_{c^*(\mu)}$, and $\tilde{A}=(\tilde{A}_t\equiv a)_{t\geq0}$ is $\Prob_{x,a}$-a.s.~the unique optimal control for the problem~(\ref{eq:control:problem:max}).
\end{enumerate}
\end{Thm}

\section{Proofs}\label{sec:proofs}
In this section, we solve the boundary problems \eqref{eq:FBP:Vc} and \eqref{eq:FBP:Wc} and prove that the solution of each is the value function of the associated control problem. We begin with the optimal expulsion problem.

\subsection{Free boundary problem for the minimization problem}

%existence of \bar{V}
%Before proving the existence of the solution of the free boundary problem~(\ref{eq:FBP:Vc}) in both cases $c=c^*:=c^*(\mu)$ and $0<c<c^*$, we need some auxiliary computations whose details can be found in~\cite[Ch.2]{Vinck_PhD}.
The general solution to the o.d.e.~(\ref{eq:ODE:continuation}) is $v(x)=-\frac{x}{\mu}+c_1e^{-2\mu x}+c_2$ where $c_1$ and $c_2$ are arbitrary constants. Since~(\ref{eq:ODE:continuation}) is satisfied on the two disjoint intervals $[0,a_c[$ and $]b_c,1]$, this yields four arbitrary constants to determine. The boundary conditions~(\ref{eq:boundary:cond}) reduce this to two unknown constants (see \cite[Ch.2]{Vinck_PhD} for details). The value of $\bar{V}_c(x,1)$ for $x\in[a_c,b_c]$ is obtained by using~(\ref{eq:switching}) together with~(\ref{eq:symmetry}). We then find that $\bar{V}_c$ must satisfy
\begin{equation}\label{eq:sol:FBP:min:bis}
\bar{V}_c(x,1)=\left\{
           \begin{array}{ll}
            \displaystyle   -\tfrac{x}{\mu}+\beta_c \left(e^{-2\mu x}-1\right), & x\in[0,a_c[\,, \\
            \displaystyle   \tfrac{x}{\mu}+\alpha_c \left(e^{2\mu x}-1\right)+ c, & x\in[a_c,b_c], \\
            \displaystyle   \tfrac{1-x}{\mu}+\alpha_c \left(e^{2\mu(1-x)}-1\right), & x\in\,]b_c,1], \\
           \end{array}
         \right.
\end{equation}
where $\alpha_c, \beta_c, a_c$ and $b_c$ are four unknowns that we have to determine using the equations~(\ref{eq:fit1})--(\ref{eq:smooth:fit2}). They give us after some simplifications, in the same order:
\begin{align}
\tfrac{-2a_c}{\mu}+ \beta_c \left(e^{-2\mu a_c}-1\right)&=\alpha_c \left(e^{2\mu a_c}-1\right)+ c,      \label{eq:resolutiona}\\
\alpha_c \left(e^{2\mu b_c}-1\right)+ c&=\tfrac{1-2b_c}{\mu}+ \alpha_c \left(e^{2\mu(1-b_c)}-1\right),  \label{eq:resolutionb}\\
2 \mu\alpha_c e^{2\mu a_c}+2\beta_c\mu e^{-2\mu a_c} + \tfrac{2}{\mu}&=0,                               \label{eq:resolutionc}\\
2 \mu\alpha_c \left(e^{2\mu b_c} + e^{2\mu(1-b_c)}\right)+  \tfrac{2}{\mu}&=0.                        \label{eq:resolutiond}
\end{align}
Multiply~(\ref{eq:resolutiona}) by $e^{2\mu a_c}$, (\ref{eq:resolutionb}) by $e^{2\mu b_c}$, (\ref{eq:resolutionc}) by $\frac{1}{2\mu}\,e^{2\mu a_c}$ and (\ref{eq:resolutiond}) by $\frac{1}{2\mu}\,e^{2\mu b_c}$, to obtain, after simplifications, respectively the four equations
\begin{align}
\alpha_c e^{4\mu a_c} +\left( \beta_c-\alpha_c+\tfrac{2a_c}{\mu}+c\right) e^{2\mu a_c}-\beta_c&=0,          \label{eq:resolutiona:bis}\\
\alpha_c e^{4\mu b_c}+\left(\tfrac{2b_c-1}{\mu}+c\right) e^{2\mu b_c}-\alpha_c e^{2\mu}&=0,             \label{eq:resolutionb:bis}\\
\alpha_c e^{4\mu a_c} + \tfrac{e^{2\mu a_c}}{\mu^2}+\beta_c&=0,                                          \label{eq:resolutionc:bis}\\
\alpha_c e^{4\mu b_c} + \tfrac{e^{2\mu b_c}}{\mu^2}+\alpha_c e^{2\mu}&=0.                                \label{eq:resolutiond:bis}
\end{align}
Subtract~(\ref{eq:resolutionb:bis}) from~(\ref{eq:resolutiond:bis}), and solve this equation for $\alpha_c$, then insert this expression into the sum of~(\ref{eq:resolutionb:bis}) and~(\ref{eq:resolutiond:bis}), to get a new equation~(\ref{eq:resolution:bc}) for $b_c$. Solve~(\ref{eq:resolutionc:bis}) for $\beta_c$, and plug this value into~(\ref{eq:resolutiona:bis}) to get a new equation~(\ref{eq:resolution:ac}) for $a_c$. These equations are
\begin{align}
e^{2 \mu(2 b_c-1)}(\mu(2 b_c-1)+c\mu^2-1)+\mu(2 b_c-1)+c\mu^2+1&=0,                          \label{eq:resolution:bc}\\
\mu^2\alpha_c e^{4\mu a_c} +(1-2\mu^2\alpha_c)e^{2\mu a_c}-2\mu a_c+\mu^2 \alpha_c-c\mu^2-1&=0,    \label{eq:resolution:ac}
\end{align}
and the formulas for $\alpha_c$ and $\beta_c$ are
\begin{equation*}
\alpha_c  =e^{2\mu b_c}\left(\tfrac{2b_c-1}{\mu}+c-\tfrac{1}{\mu^2}\right)\tfrac{1}{2}e^{-2\mu},    \qquad% \label{eq:resolution:alpha}
\beta_c   =-\alpha_c e^{4\mu a_c}-\tfrac{e^{2\mu a_c}}{\mu^2}.                                         %\label{eq:resolution:beta}
\end{equation*}
Equations~(\ref{eq:resolution:bc}) for $b_c$ and~(\ref{eq:resolution:ac}) for $a_c$ are transcendental. We define
\begin{align}
h_c(t)          &=e^{2 t} ( t+c\mu^2-1) + t+c\mu^2+1,\label{eq:def:hct}\\
\tilde{h}_c(s)  &=\mu^2\alpha_c e^{2 s}+ (1-2\mu^2\alpha_c)e^s -s +\mu^2\alpha_c -c\mu^2-1,\label{eq:def:hctildes}
\end{align}
so that $b_c$ and $a_c$ are respectively  solutions of $h_c(\mu(2 b_c-1))=0$ and  $\tilde{h}_c(2\mu a_c)=0$. Setting $t_c=\mu(2 b_c-1)$ and $s_c=2\mu a_c$, we find that solving the system~(\ref{eq:resolutiona})--(\ref{eq:resolutiond}) with $0<a_c\leq b_c<\frac{1}{2}$ is equivalent to solving
\begin{align}
h_c(t_c)        &=0,\label{eq:hct=0}\\
\tilde{h}_c(s_c)&=0 ,\label{eq:hctildes=0}\\
\alpha_c          &=(t_c+c\mu^2-1)\tfrac{1}{2\mu^2}\,e^{t_c-\mu},\label{eq:alpha:2}\\
\beta_c           &=-\alpha_c e^{2s_c}-\tfrac{1}{\mu^2}\, e^{s_c},\label{eq:beta:2}
\end{align}
with $-\mu<t_c<0$ and $0<s_c \leq t_c+\mu$.

\begin{Rem}\label{rem:discussion:h:htilde}
By computing two derivatives, we see that $h_c^{\prime\prime}(t)>0$, for all $t\in\R$, $\lim_{t\downarrow-\infty}h_c^\prime(t)=1$, therefore $h_c^\prime(t)>0$, for all $t\in\R$, therefore $h_c$ is strictly increasing with $h_c(-c\mu^2-1)=-2e^{-2(c\mu^2+1)}<0$ and $h_c(-c\mu^2)=-e^{-2c\mu^2}+1>0$, so that~(\ref{eq:hct=0}) admits a unique solution $t_c\in\,]-c\mu^2-1,-c\mu^2[$. On the other hand, $\tilde{h}_c(0)=-c\mu^2<0$ and we see by direct calculation that $\tilde{h}_c^\prime(s)=0$ if and only if $s\in \left\{0, \log\left(\frac{-1}{2\mu^2\alpha_c}\right)\right\}$. We set
\begin{equation}\label{eq:def:mtilde}
\tilde{m}_c=\log\left(\frac{-1}{2\mu^2\alpha_c}\right)
\end{equation}
and depending on the sign of this value and of $\tilde{h}_c\left(\tilde{m}_c\right)$,~(\ref{eq:hctildes=0}) has up to three solutions. This will be discussed later on a case by case basis.
\end{Rem}

%\subsubsection{Explicit resolution in the case where $c=c^*(\mu)$}
\begin{proof}[Proof of Proposition~\ref{prop:sol:FBP:min}]
Let us start with the case where $c=c^* := c^*(\mu)$. In this case, the unique solution of~(\ref{eq:hct=0}) is given by
\begin{equation}\label{eq_def_tc*_1}
t_{c^*}= -c^*\mu^2 -\sqrt{1-\tfrac{\mu^2}{\sinh^2(\mu)}}.
\end{equation}
Indeed, by \eqref{eq:def:c*}, \eqref{eq_def_tc*_1} and direct computations,
\begin{align*}
h_{c^*}(t_{c^*})&= e^{2 t_{c^*}} ( t_{c^*}+c^*\mu^2-1) + t_{c^*}+c^*\mu^2+1 \\
&= \left(\tfrac{\sinh(\mu)}{\mu}\left(1-\sqrt{1-\tfrac{\mu^2}{\sinh^2(\mu)}}\right) \right)^2\left( -\sqrt{1-\tfrac{\mu^2}{\sinh^2(\mu)}}-1\right) -\sqrt{1-\tfrac{\mu^2}{\sinh^2(\mu)}}+1\\
&=0.
\end{align*}
Using the formula for $c^*$ in~(\ref{eq:def:c*}), we can write
\begin{equation}\label{eq_def_tc*_2}
t_{c^*}=\log\left(\tfrac{\sinh(\mu)}{\mu}\left(1-\sqrt{1-\tfrac{\mu^2}{\sinh^2(\mu)}}\right) \right).
\end{equation}
Moreover, $0>t_{c^*}>-\mu$. Indeed, the first inequality follows from the fact that $\sinh \mu>\mu$ and the second one is equivalent to
$\tfrac{\sinh(\mu)}{\mu}\left(1-\sqrt{1-\tfrac{\mu^2}{\sinh^2(\mu)}}\right)>e^{-\mu}$, which, in turn, is equivalent to
$e^{2\mu}(\mu-1)+\mu+1>0$; this last inequality is satisfied for all $\mu>0$. Plugging into~(\ref{eq:alpha:2}) the value of $t_{c^*}$ given by~(\ref{eq_def_tc*_1}), or by~(\ref{eq_def_tc*_2}) when it appears in an exponential, we obtain
\begin{equation}\label{eq:alpha*}
\alpha_{c^*}=\frac{-e^{-\mu}}{2\mu\sinh(\mu)}.
\end{equation}
We now observe that the unique solution $s_{c^*}$ of~(\ref{eq:hctildes=0}) such that $0<s_{c^*}\leq t_{c^*}+\mu$ is given by
\begin{equation}\label{eq:s0*}
s_{c^*}= t_{c^*}+\mu.
\end{equation}
Indeed, by definition of $t_{c^*}$ (see~(\ref{eq_def_tc*_1}) or~(\ref{eq_def_tc*_2})) and of $\alpha_{c^*}$ in~(\ref{eq:alpha*}), we have
\begin{align*}
\tilde{h}_{c^*}(t_{c^*}+\mu)&=\mu^2\alpha_{c^*} e^{2 (t_{c^*}+\mu)}+ (1-2\mu^2\alpha_{c^*})e^{t_{c^*}+\mu}-(t_{c^*}+\mu) +\mu^2\alpha_{c^*}  -c^*\mu^2-1\\
&=\mu^2\alpha_{c^*} \left(e^{t_{c^*}+\mu}-1\right)^2+ e^{t_{c^*}+\mu}  -(t_{c^*}+c^*\mu^2) -\mu-1\\
&=\tfrac{-\mu e^{-\mu}}{2\sinh(\mu)} \left[e^{\mu}\left(\tfrac{\sinh(\mu)}{\mu}\left(1-\sqrt{1-\tfrac{\mu^2}{\sinh^2(\mu)}}\right) \right)-1\right]^2\\
&\hspace{4ex} +e^{\mu}\left[\tfrac{\sinh(\mu)}{\mu}\left(1-\sqrt{1-\tfrac{\mu^2}{\sinh^2(\mu)}}\right) \right] +\sqrt{1-\tfrac{\mu^2}{\sinh^2(\mu)}} -\mu-1\\
&=0.
\end{align*}
The uniqueness on the interval $]0,t_{c^*}+\mu]$ is obtained by considering the position of the extrema of $\tilde{h}_{c^*}$. Indeed, setting
$$\tilde{m}_{c^*}=\log\left(\frac{-1}{2\mu^2\alpha_{c^*}}\right)=\log\left(e^\mu\,\tfrac{\sinh(\mu)}{\mu}\right)=\mu+\log\left(\tfrac{\sinh(\mu)}{\mu}\right)$$
and since $\sinh(\mu)/\mu>1$ for all $\mu>0$, we have that $\tilde{m}_{c^*}>\mu>t_{c^*}+\mu>0$. Thus, the function $\tilde{h}_{c^*}(s)$ vanishes three times: the first time on the interval $]-\infty, 0[\,$, then it reaches at $0$ a local minimum $\tilde{h}_{c^*}(0)=-c^*\mu^2$, then it increases, vanishes at $t_{c^*}+\mu<\tilde{m}_{c^*}$ and keeps increasing until $\tilde{m}_{c^*}$ where it reaches a local maximum. Therefore $\tilde{h}_c^*(\tilde{m}_{c^*})>0$. Finally, $\tilde{h}_{c^*}(s)$ vanishes a third time on the interval $]\tilde{m}_{c^*},+\infty[$.

It remains to determine the parameter $\beta_{c^*}$ which, by~(\ref{eq:beta:2}), (\ref{eq:s0*}), (\ref{eq_def_tc*_2}) and~(\ref{eq:alpha*}), is given by
\begin{equation}\label{eq:def:beta*}
\beta_{c^*}=\tfrac{-e^{\mu}}{2\mu\sinh(\mu)}.
\end{equation}
We have therefore solved the system~(\ref{eq:hct=0})--(\ref{eq:beta:2}) and we have found that
\begin{equation}\label{eq:ac*=bc*}
b_{c^*}=\frac{1}{2}\left(\frac{t_{c^*}}{\mu}+1\right)\in\,\left]0,\tfrac{1}{2}\right[ \quad\text{ and }\quad a_{c^*}=\frac{s_{c^*}}{2\mu}=\frac{t_{c^*}+\mu}{2\mu}=b_{c^*},
\end{equation}
with $t_{c^*}$ given by~(\ref{eq_def_tc*_1}). This establishes the existence and uniqueness of the solution $\bar{V}_{c^*}$ of the free boundary problem~(\ref{eq:FBP:Vc}). This solution is given by
\begin{align}
\bar{V}_{c^*}(x,1)&=\left\{
           \begin{array}{ll}
           \displaystyle  \tfrac{-x}{\mu}+\beta_{c^*} \left(e^{-2\mu x}-1\right), & x\in[0,a_{c^*}], \\
           \displaystyle  \tfrac{1-x}{\mu}+\alpha_{c^*} \left(e^{2\mu(1-x)}-1\right), & x\in[a_{c^*},1], \\
           \end{array}
         \right. \nonumber\\
&=\left\{
           \begin{array}{ll}
            \displaystyle \tfrac{-x}{\mu}+\tfrac{1-e^{-2\mu x}}{\mu(1-e^{-2\mu})}\,, &\qquad\qquad x\in[0,a_{c^*}], \\
           \displaystyle  \tfrac{-x}{\mu}+\tfrac{1-e^{-2\mu x}}{\mu(1-e^{-2\mu})}\,, &\qquad\qquad x\in[a_{c^*},1], \\
           \end{array}
         \right.\label{eq:value:barVc*} \\
&=f^\mu(x)=\E_x\left(\sigma^\mu\right), \quad\qquad\qquad\; x\in[0,1],\nonumber
\end{align}
and $\bar{V}_{c^*}(x,-1)=\bar{V}_{c^*}(1-x,1)=f^{-\mu}(x)$. By definition of $\tilde{A}$, we have $f^{a\mu}=J_{c^*}(x,a,\tilde{A})$ where $\tilde{A}$ is the constant strategy $(\tilde A_t \equiv a)$. Finally, we find using~(\ref{eq:ac*=bc*}) and~(\ref{eq_def_tc*_2}) that
$$a_{c^*}=b_{c^*}=\frac{1}{2\mu}\left(\log\left(\tfrac{\sinh(\mu)}{\mu}\left(1-\sqrt{1-\tfrac{\mu^2}{\sinh^2(\mu)}}\right)\right)+\mu\right),$$
which, according to~(\ref{eq:maximum:diff:fnu}), is the location of the global maximum of $f^\mu(x)-f^{-\mu}(x)$ on the interval $[0,1]$. Therefore, by definition of $c^*(\mu)$, we have $f^\mu(a_{c^*})=f^{-\mu}(a_{c^*})+c^*(\mu)$, and this completes the proof of the first part of Proposition~\ref{prop:sol:FBP:min}.

Let us now consider $0<c<c^*(\mu)$. The form of the solution $\bar{V}_c$ of the free boundary problem~(\ref{eq:FBP:Vc}) given in~(\ref{eq:sol:FBP:min:c<c*}) has already been discussed starting with~(\ref{eq:sol:FBP:min:bis}) and reduced to the resolution of the equivalent system~(\ref{eq:hct=0})--(\ref{eq:beta:2}) with $-\mu<t_c<0$ and $0<s_c<t_c+\mu$.

We start by showing that the unique solution $t_c$ of~(\ref{eq:hct=0}) mentioned in Remark~\ref{rem:discussion:h:htilde} is such that $-\mu<t_{c^*}<t_c<-c\mu^2$. The first inequality is mentioned just after~(\ref{eq_def_tc*_2}) and the last inequality, as well as the uniqueness of the solution, has been discussed in Remark~\ref{rem:discussion:h:htilde}. Observe that by definition, $t\mapsto h_c(t)$ and $c\mapsto h_c(t)$ are both strictly increasing. Thus, $h_c(t_{c^*})<h_{c^*}(t_{c^*})=0$ and since, by definition, $h_c(t_c)=0$, we obtain the last inequality $t_{c^*}<t_c$ and this establishes~(\ref{eq:bc}) of Proposition~\ref{prop:sol:FBP:min}.

We now establish~(\ref{eq:alpha}). Notice that by definition of $t_c$, we have $e^{2 t_c} ( t_c+c\mu^2-1) + t_c+c\mu^2+1=0$, which is equivalent to
\begin{equation}\label{eq:t0:sinh:cosh}
t_c + c\mu^2=\frac{-1+e^{2 t_c}}{1+e^{2 t_c}}=\frac{\sinh(t_c)}{\cosh(t_c)}
\end{equation}
and which yields~(\ref{eq:alpha}). Indeed, by~(\ref{eq:alpha:2}),
\begin{equation}\label{eq:alpha:bis}
\alpha_c=(t_c+c\mu^2-1)\frac{1}{2\mu^2}e^{t_c-\mu}=\frac{-e^{-\mu}}{2\mu^2}\frac{1}{\cosh(t_c)}=\frac{-e^{-\mu}}{2\mu^2\cosh(2\mu b_c-\mu)}\,.
\end{equation}

In order to prove~(\ref{eq:ac}) or, equivalently, to show that there exists a unique solution $s_c$ of~(\ref{eq:hctildes=0}) such that $0<s_c<t_c+\mu$, we need to study the function
$\tilde{h}_c$ and more particularly the position of its local maximum $\tilde{h}_c(\tilde{m}_c)$ (see Remark~\ref{rem:discussion:h:htilde}). Observe that by~(\ref{eq:def:mtilde}) and~(\ref{eq:alpha:bis}),
$$\tilde{m}_c=\log\left(\frac{-1}{2\mu^2\alpha_c}\right) = \mu +\log\left(\cosh(t_c)\right)>\mu.$$
As $t_c<0$, we have
\begin{equation}\label{eq_inequ_2mubc_mtilde}
0<2\mu b_c=t_c+\mu < \tilde{m}_c.
\end{equation}
Plugging $s=2\mu b_c=t_c+\mu$ into~(\ref{eq:def:hctildes}), we see that
$$\tilde{h}_c(t_c+\mu)=\mu^2\alpha_c\left(e^{t_c+\mu}-1\right)^2+e^{t_c+\mu}-\left(t_c+c\mu^2\right)-\mu-1.$$
Since $\mu^2\alpha_c=-\frac{e^{-\mu}}{e^{t_c}+e^{-t_c}}$ by~(\ref{eq:alpha}), and expressing $t_c+c\mu^2$ using~(\ref{eq:t0:sinh:cosh}), we obtain
\begin{align*}
\tilde{h}_c(t_c+\mu)=\frac{-e^{-\mu}}{e^{t_c}+e^{-t_c}}\Big[ &\left(e^{t_c+\mu}-1\right)^2 - e^{t_c+2\mu}\left(e^{t_c}+e^{-t_c}\right)  + e^{t_c+\mu}-e^{-t_c+\mu} \\
                                                             &\quad+ (\mu+1)\left(e^{t_c+\mu} + e^{-t_c+\mu}\right)  \Big],
\end{align*}
and this simplifies to
\begin{equation}\label{eq:htilde:tc+mu}
\tilde{h}_c(t_c+\mu)= -\mu +\frac{\sinh(\mu)}{\cosh(t_c)}\,.
\end{equation}
Define
$$k(t)=-\mu +\frac{\sinh(\mu)}{\cosh(t)}\,.$$
Clearly, this function has two zeros $\tilde{t}<0$ and $-\tilde{t}>0$. From~(\ref{eq:htilde:tc+mu}),
$k(t_{c^*})=\tilde{h}_{c^*}(t_{c^*}+\mu)=0$
by~(\ref{eq:s0*}). Therefore, $\tilde{t}=t_{c^*}$ and, since $t_{c^*}<t_c<0$, we conclude that $k(t_c)>0$, which implies that
\begin{equation}\label{eq:htilde:tc+mu:positive}
\tilde{h}_c(2\mu b_c)=\tilde{h}_c(t_c+\mu)=k(t_c)>0.
\end{equation}
Since $\tilde{h}_c$ is monotone on $[0,\tilde{m}_c]$ and $\tilde{h}_c(0)=-c\mu^2<0$,  we conclude that $0<s_c<t_c+\mu$, which establishes~(\ref{eq:ac}).
Finally,~(\ref{eq:beta}) is simply a rewriting of~(\ref{eq:beta:2}).
\end{proof}

\subsection{Proof of optimality for the minimization problem}
We now aim to prove that the solution $\bar{V}_c$ of \eqref{eq:FBP:Vc} given in Proposition \ref{prop:sol:FBP:min} is indeed the value function. For this, we could apply the Verification Theorem 6.2 of~\cite{OksendalSulemBook}. However, we prefer to give a direct proof, in which the main steps correspond to checking assumptions of the Verification Theorem. Indeed, Lemma~\ref{lem:generator:Vbar} below corresponds to verifying hypotheses (vi) and (x) of \cite[Theorem 6.2]{OksendalSulemBook}, and Corollary \ref{cor:DeltaH} below corresponds to hypothesis (ii) there. However, by detailing the proof, we can identify where any deviation from the strategy $A^c$ leads to a suboptimal cost function, and this will be useful in our result on uniqueness (Proposition \ref{prop:uniqueness} in Section \ref{sec:further}).

Before proving Proposition~\ref{prop:esperance:tauAc}, we introduce some notations. Let $(\tau_n)$ be the sequence of switching times of $A^c$ and let $\tau^{A^c}$ be the exit time of $X^{A^c}$ from $[0,1]$. For a Brownian motion $B$ starting a.s.~at $x$ and for $a<x<b$, let
\begin{equation}\label{eq:def:pxab}
p^{\pm}_x(a,b)=\Prob_x\left\{\mbox{$B_t\pm\mu t$ hits $a$ before $b$}\right\}.
\end{equation}
These quantities satisfy the translation invariance property $p^{\pm}_x(a,b)=p^{\pm}_{x+h}(a+h,b+h)$ for all $h\in\R$ and the symmetry property $p^-_x(a,b)=p^+_{-x}(-a,-b)$. We also define
\begin{align}
\sigma^{\pm}_{a,b}   &=\inf\{t\geq0: B_t\pm\mu t\notin\,]a,b[\},\nonumber\\
E_x(y,a,b)              &=\E\left(\sigma^+_{a,b}\left\vert\, B_{\sigma^+_{a,b}}+\mu \sigma^+_{a,b}=y, B_0=x\right.\right),\qquad y\in\{a,b\}.\label{eq:def:Eyab}
\end{align}
These expectations and probabilities can all be explicitly computed (see e.g.~\cite[II.2.3]{Borodin_salminen02}). In particular, the expectations are finite and $0< p^{\pm}_x(a,b) <1$.

\begin{proof}[Proof of Proposition~\ref{prop:esperance:tauAc}]
Using the above notations, we find that for all $k\in\N$,
\begin{align*}
\Prob_{x,a}\left(N_{\tau^{A^c}}(A^c)=k\right)=&\,\Prob_{x,a}\left(\tau_{k}<\tau^{A^c}<\tau_{k+1}\right)\\
=&\,\Prob_{x,a}\left(\tau^{A^c}<\tau_{k+1}\mid \tau_{k}<\tau^{A^c}\right)\Prob_{x,a}\left(\tau_{k}<\tau^{A^c}\mid \tau_{k-1}<\tau^{A^c}\right)\\
&\qquad\times\cdots\times\Prob_{x,a}\left(\tau_{2}<\tau^{A^c}\mid \tau_{1}<\tau^{A^c}\right)\Prob_{x,a}\left(\tau_{1}<\tau^{A^c}\right)
\end{align*}
since $\{\tau_j<\tau^{A^c}\}\supset\{\tau_{j+1}<\tau^{A^c}\}$ for all $j\geq1$.

If $x\in\{0,1\}$, then $\E_{x,a}\left(\tau^{A^c}\right)=\E_{x,a}\left(N_{\tau^{A^c}}(A^c)\right)=0$. Consider now the case where $x\in\,]b_c,1[$ and $A_{0-}^c=a=1$. Notice that for all $k\geq1$, on $\{\tau_k<\tau^{A^c}\}$, $X^{A^c}_{\tau_k}=b_c$ if $k$ is odd and $X^{A^c}_{\tau_k}=1-b_c$ if $k$ is even. Clearly, $\Prob_{x,1}\left(\tau_{1}<\tau^{A^c}\right)=p^+_x(b_c,1)$ and since the process $X^{A^c}$ is strong Markov by construction, for $k$ even (but also, by symmetry, for $k$ odd),
\begin{equation*}
\Prob_{x,1}\left(\tau_{k+1}<\tau^{A^c}\mid \tau_k<\tau^{A^c}\right)=\Prob_{1-b_c,1}\left(\tau_1<\tau^{A^c}\right)=p^+_{1-b_c}(b_c,1).
%\\
%\Prob_{x,1}\left(\tau^{A^c}<\tau_{k+1}\mid \tau_k<\tau^{A^c}\right)&=\Prob_{1-b_c,1}\left(\tau_1>\tau^{A^c}\right)=1-p^+_{1-b_c}(b_c,1).
\end{equation*}
Therefore,
\begin{equation}\label{eq:P(N=k)}
\Prob_{x,1}\left(N_{\tau^{A^c}}(A^c)=k\right)=\left\{ \begin{array}{ll}
                                                        1-p^+_x(b_c,1), & \text{ if } k=0, \\
                                                       p^+_x(b_c,1)\left(p^+_{1-b_c}(b_c,1)\right)^{k-1}(1-p^+_{1-b_c}(b_c,1)), & \text{ if } k\geq1,
                                                      \end{array}
\right.
\end{equation}
that is, given $N_{\tau^{A^c}}(A^c)\geq1$, $N_{\tau^{A^c}}(A^c)$ is a geometric r.v.~with parameter $1-p^+_{1-b_c}(b_c,1)\in\,]0,1[\,$. Therefore,
$\E_{x,1}\left(N_{\tau^{A^c}}(A^c)\right)<+\infty$, and this establishes the second statement in Proposition~\ref{prop:esperance:tauAc}.

Turning to the other statement, by the law of total probability,
\begin{equation}\label{eq:serie:E:tauAc}
\E_{x,1}\left(\tau^{A^c}\right)=\sum_{k=0}^{+\infty}\E_{x,1}\left(\tau^{A^c}\left\vert\, \tau_{k}<\tau^{A^c}<\tau_{k+1}\right.\right)\Prob_{x,1}\left(N_{\tau^{A^c}}(A^c)=k\right).
\end{equation}
For $k=0$,
$$\E_{x,1}\left(\tau^{A^c}\left\vert\, 0<\tau^{A^c}<\tau_{1}\right.\right)=E_x(1,b_c,1).$$
For $k\geq1$, on $\{ \tau_k<\tau^{A^c}<\tau_{k+1}\}$,
$$\tau^{A^c}=\tau_1+(\tau_2-\tau_1)+\cdots+(\tau_k-\tau_{k-1})+  (\tau^{A^c}-\tau_k).$$
By the strong Markov property,
$$\E_{x,1}\left(\tau_1\left\vert\, \tau_k<\tau^{A^c}<\tau_{k+1}\right.\right)=\E_{x,1}\left(\tau_1\left\vert\, \tau_1=\sigma^+_{b_c,1}\right.\right)=E_x(b_c,b_c,1),$$
and for $2\leq \ell\leq k$,
$$\E_{x,1}\left(\tau_\ell-\tau_{\ell-1}\left\vert\, \tau_k<\tau^{A^c}<\tau_{k+1}\right.\right)=E_{1-b_c}(b_c,b_c,1)$$
and
$$\E_{x,1}\left(\tau^{A^c}-\tau_k\left\vert\, \tau_k<\tau^{A^c}<\tau_{k+1}\right.\right)=E_{1-b_c}(1,b_c,1).$$
Therefore, for $k\geq1$,
$$\E_{x,1}\left(\tau^{A^c}\left\vert\, \tau_k<\tau^{A^c}<\tau_{k+1}\right.\right)=E_x(b_c,b_c,1)+(k-1)E_{1-b_c}(b_c,b_c,1)+E_{1-b_c}(1,b_c,1),$$
and we conclude from~(\ref{eq:P(N=k)}) and~(\ref{eq:serie:E:tauAc}) that $\E_{x,1}\left(\tau^{A^c}\right)<+\infty$.

The cases $x\in[0,a_c[\,$, $x\in[a_c,b_c]$ and $a=-1$ are treated similarly. This completes the proof of Proposition~\ref{prop:esperance:tauAc}.
\end{proof}

\begin{Lem}\label{lem:generator:Vbar}
Let ${\mathbb L}_a$ denote the infinitesimal generator of a Brownian motion with drift $a\mu$, $a\in\{\pm1\}$; that is, for $f\in \mathcal{C}^2([0,1], \R)$
\begin{equation}\label{eq:def:generator}
{\mathbb L}_af(x)=a\mu \frac{\partial f}{\partial x}(x)+\frac{1}{2}\frac{\partial^2 f}{\partial x^2}(x).
\end{equation}
Then, for all $0<c<c^*(\mu)$,
\begin{align}
1+{\mathbb L}_a \bar{V}_c(x, a)&=0,     \qquad\mbox{ for all } x \in  C_a\label{eq:generator:Ca},\\
1+{\mathbb L}_a \bar{V}_c(x, a)&>0,     \qquad\mbox{ for all } x \in  D_a\setminus \partial D_a\label{eq:inequ:generator:Da}.
\end{align}
Moreover, if $c=c^*(\mu)$, then
\begin{equation}\label{eq:generator:Ca:c*}
1+{\mathbb L}_a \bar{V}_{c^*}(x, a)=0,     \qquad\mbox{ for all } x \in  [0,1].
\end{equation}
\end{Lem}
\begin{proof}
In the case where $c=c^*(\mu)$, (\ref{eq:generator:Ca:c*}) is a standard property of $\bar{V}_{c^*}(x, a)=f^{a\mu}(x)$ (see~(\ref{eq:sol:FBP:c*:V})), but it is also easily obtained from the explicit expression that we have for $f^{a\mu}$ (see~(\ref{eq:def:fnu})).

Assume $0<c<c^*$ and consider the case where $a=1$, since the other case is similar. The function $\bar{V}_c(x,1)$ is $\mathcal{C}^2$ in both $C_1$ and $D_1\setminus \partial D_1$. The first equality~(\ref{eq:generator:Ca}) follows from the construction of $\bar{V}_c$ (see~(\ref{eq:ODE:continuation})). It remains to prove (\ref{eq:inequ:generator:Da}). By construction (see~(\ref{eq:switching})), for all $x\in D_1\setminus \partial D_1=\,]a_c,b_c[\,$, we have $\bar{V}_c(x,1)=c+\bar{V}_c(x,-1)$  and thus,
$$1+{\mathbb L}_1 \bar{V}_c(x, 1)=1+{\mathbb L}_1 (c+\bar{V}_c(x,-1))=1+\mu\,\frac{\partial \bar{V}_c}{\partial x}(x,-1)+\frac{1}{2}\,\frac{\partial^2 \bar{V}_c}{\partial x^2}(x,-1).$$
Moreover, since $]a_c,b_c[$ belongs to $ C_{-1}$ in which~(\ref{eq:generator:Ca}) is satisfied, we have
\begin{equation}\label{3.38a}
1+{\mathbb L}_1 \bar{V}_c(x, 1)=1+\mu\,\frac{\partial \bar{V}_c}{\partial x}(x,-1)+\frac{1}{2}\,\frac{\partial^2 \bar{V}_c}{\partial x^2}(x,-1)-(1+{\mathbb L}_{-1} \bar{V}_c(x, -1))=2\mu\,\frac{\partial \bar{V}_c}{\partial x}(x,-1).
\end{equation}
Thus, we have to prove that $\mu\,\frac{\partial \bar{V}_c}{\partial x}(x,-1)>0$ for all $x\in D_1\setminus \partial D_1$. If $x\in\,]a_c,b_c[\,$, then $1-x\in\,]1-b_c,1-a_c[\,\subset\,]b_c,1]$ and according to~(\ref{eq:sol:FBP:min:c<c*}),
$$\bar{V}_c(x, -1)=\bar{V}_c(1-x, 1)=\tfrac{x}{\mu}+\alpha_c \left(e^{2\mu x}-1\right).$$
Therefore, $\mu\,\frac{\partial \bar{V}_c}{\partial x}(x,-1)=1+ 2\mu^2\alpha_c e^{2\mu x}$ for all $x\in\,]a_c,b_c[\,$.  As $x\in \,]a_c,b_c[\,\subset \,]0,\frac{1}{2}[$ and $\alpha_c<0$, we get according to~(\ref{eq:alpha:bis}) that
%\begin{equation}\label{eq:ineq:derivative:barV}
$$1+ 2\mu^2\alpha_c e^{2\mu x}>1+ 2\mu^2\alpha_c e^{\mu} =1-\frac{1}{\cosh(2\mu b_c-\mu)}\,>0,$$
%\end{equation}
which establishes the lemma.
\end{proof}

\begin{Lem}\label{lem:ineq:Vbar}
For all $0<c\leq c^*(\mu)$ and any $a\in\{-1,1\}$, if $x\in C_a$, then
$$\bar{V}_c(x,a)<c+\bar{V}_c(x,-a).$$
\end{Lem}

\begin{proof}
In the case where $c=c^*(\mu)$, the result follows immediately from~(\ref{eq:ac*:argmax}) together with~(\ref{eq:sol:FBP:c*:V}). Assume $0<c<c^*(\mu)$ and consider the case where $a=1$, since the other case is similar. We distinguish four cases.
\vskip 10pt

 {\em Case 1.} If $x\in[0,a_c[\,$, then $1-x\in\,]1-a_c,1]\subset\,]b_c,1]$ and, according to~(\ref{eq:symmetry}) and~(\ref{eq:sol:FBP:min:c<c*}),
$$c+\bar{V}_c(x,-1)-\bar{V}_c(x,1)=c+ \tfrac{2x}{\mu} +\alpha_c\left(e^{2\mu x}-1\right) -\beta_c\left(e^{-2\mu x}-1\right)=: d_1(x).$$
%Let $d_1(x)$ denote the expression on the right-hand side.
We will show that $d_1(x)>0$ for $x\in[0,a_c[\,$. Indeed, $d^\prime_1(x)$ vanishes at most twice on $\R$, because the equation $d_1'(x)=0$ is a quadratic equation in $ e^{2\mu x}$. The roots of $d^\prime_1(x)$ are given by
\begin{equation}\label{eq:def:roots:x1pm}
x_1^{\pm}=\frac{1}{2\mu}\log\left(-\frac{1}{2\mu^2\alpha_c}\pm\sqrt{\Delta_1}\right),
\end{equation}
with $\Delta_1=\frac{1}{4\mu^4\alpha_c^2}-\frac{\beta_c}{\alpha_c}$.  By~(\ref{eq:beta}) and~(\ref{eq:alpha}), we have
$$1-4\mu^4\alpha_c\beta_c=\left(1+ 2\mu^2\alpha_ce^{2\mu a_c}\right)^2=\left(1- \frac{ e^{2\mu a_c-\mu}}{\cosh(2\mu b_c-\mu)}\right)^2>0$$
and thus $\Delta_1>0$. Using \eqref{eq:beta}, \eqref{eq:alpha} and the fact that $a_c < b_c$, we see that $\beta_c <0$, and so $(-2\mu \alpha_c)^{-1} - \sqrt{\Delta_1} >0$. Therefore, $x_1^{\pm}\in\R$. From the formula for $d^\prime_1(x)$ and~(\ref{eq:beta}), we see that
$$d_1'(a_c)=\tfrac{2}{\mu}+2\mu\alpha_c e^{2\mu a_c}+2\mu\beta_c e^{-2\mu a_c}=\tfrac{2}{\mu}+2\mu\alpha_c e^{2\mu a_c}- 2\mu \left(\alpha_c e^{4\mu a_c}+\frac{e^{2\mu a_c}}{\mu^2}\right)e^{-2\mu a_c}=0.$$
Therefore, $a_c\in\{x_1^-,x_1^+\}$. By (\ref{eq_inequ_2mubc_mtilde}),
$$0<2\mu a_c<2\mu b_c< \tilde{m}_c=\log\left(\frac{-1}{2\mu^2\alpha_c}\right).$$
Since $2\mu x_1^-<\tilde{m}_c<2\mu x_1^+$, we find that $x_1^-=a_c<x_1^+$. This implies that the function $d_1$ is monotone decreasing on $]-\infty, a_c[$ with $d_1(0)=c>0$ and $d_1(a_c)=0$ by (\ref{eq:switching}). The function $d_1$ is thus strictly positive on $[0,a_c[\,$, which proves the desired inequality.
\vskip 10pt

  {\em Case 2.} If $x\in\,]b_c,1-b_c[\,$, then $1-x\in \,]b_c,1-b_c[$ and according to~(\ref{eq:symmetry}) and~(\ref{eq:sol:FBP:min:c<c*}), we have
\begin{equation}\label{eq:d2}
c+\bar{V}_c(x,-1)-\bar{V}_c(x,1)=c +\tfrac{2x-1}{\mu}+ \alpha_c\left(e^{2\mu x}-e^{2\mu}e^{-2\mu x}\right)=: d_2(x).
\end{equation}
%Let $d_2(x)$ denote the expression on the right-hand side.
We will show that $d_2(x) >0$ for $x \in\,]b_c,1-b_c[$. Indeed, its derivative $d_2'$ vanishes at most twice on $\R$, at
$$x_2^\pm=\frac{1}{2\mu}\log\left(-\frac{1}{2\mu^2\alpha_c}\pm\sqrt{\Delta_2}\right),$$
where $\Delta_2=\frac{1}{4\mu^4\alpha_c^2}-e^{2\mu}$.  By~(\ref{eq:alpha}), we find that $-2\mu^2\alpha_c=\frac{e^{-\mu}}{\cosh(2\mu b_c-\mu)}<e^{-\mu}$ since $b_c<\frac{1}{2}$. Hence, $\Delta_2>0$, and according to~(\ref{eq:switching}) and condition~(\ref{eq:smooth:fit2}), we have that $d_2'(b_c)=0$. Moreover, we notice that $d_2'(1-b_c)=d_2'(b_c)=0$. Thus, $x_2^-=b_c<x_2^+=1-b_c$ and the function $d_2$ is monotone increasing on $]b_c,1-b_c[$. From~(\ref{eq:switching}), we have that $d_2(b_c)=0$, thus, the function $d_2$ is strictly positive on $]b_c,1-b_c[\,$, which proves the desired inequality.
\vskip 10pt

   {\em Case 3.} If $x\in[1-b_c, 1-a_c]$, then $1-x\in[a_c,b_c]=D_1$ and by~(\ref{eq:symmetry}) and~(\ref{eq:switching}),
\begin{equation}\label{e3.41a}
c+\bar{V}_c(x,-1)-\bar{V}_c(x,1)= c+\bar{V}_c(1-x,-1)+c-\bar{V}_c(x,1)= 2 c>0,
\end{equation}
which proves the desired inequality.
\vskip 10pt

   {\em Case 4.} If $x\in\,]1-a_c,1]$, then $1-x\in [0,a_c[$ and we have, according to~(\ref{eq:symmetry}) and~(\ref{eq:sol:FBP:min:c<c*}),
$$c+\bar{V}_c(x,-1)-\bar{V}_c(x,1)=c -2\frac{1-x}{\mu}+ \beta_c\left(e^{-2\mu(1-x)}-1\right)-\alpha_c\left(e^{2\mu(1-x)}-1\right)=: d_3(x).$$
%Let $d_3(x)$ denote the expression on the right-hand side.
We see that $d_3(x)=2c-d_1(1-x)$. Thus, $d_3$ is monotone on $]1-a_c,1]$ and since $d_3(1-a_c)=2c-d_1(a_c)=2c$ and $d_3(1)=2c-d_1(0)=c$, the function $d_3$ is strictly positive on this interval. This proves the desired inequality in this last case.
\end{proof}

Let us define $H^c: \R_+ \times [0,1] \times\{-1,1\}\times\N \longrightarrow \R_+$ by
\begin{equation}\label{eq:def:H}
%\begin{array}{lrll}
%\\
%&
(t, x, a, n) \mapsto H^c(t,x,a,n) := t + cn + \bar{V}_c(x, a).
%\\
%\end{array}
\end{equation}
For $A\in\mathcal{A}$, we consider the process $H^{c,A}_t=H^c(t,X^A_t, A_t, N_t(A))$, where $N_t(A)$ is given by~(\ref{eq:def:NtA}), and we write $\Delta H^{c,A}_s=H^{c,A}_s-H^{c,A}_{s-}$. Observe that $H^{c,A}_{0-}=\bar{V}_c(x, a)$, $\Prob_{x,a}$-almost surely.

\begin{Cor}\label{cor:DeltaH}
Let $0<c\leq c^*(\mu)$, let $A^c$ denote the control satisfying~(\ref{eq:def:candidate:Ac}) and let $A\in\mathcal{A}$ be any admissible control. Then, for all $x\in[0,1]$ and $a\in\{-1,1\}$, $\Prob_{x,a}$-a.s., for all $s\in\R_+$:
%\begin{enumerate}  \item
(1) $\Delta H^{c,A^c}_s=0$;
%  \item
(2) $\Delta H^{c,\,A}_s\geq0$;
%  \item
and (3) $\Delta H^{c,\,A}_s>0$  if and only if $X^A_s\in C_{A_{s-}}$ and $A_s\neq A_{s-}$.
%\end{enumerate}
\end{Cor}
\begin{proof}
Let $A\in\mathcal{A}$ and let  $s\in\R_+$. If $s$ is a time of continuity of the strategy $A$, then clearly $\Delta H^{c,\,A}_s=0$. Assume $s$ is such that $A_{s-}\neq A_s$. Then $\Delta N_s(A)=1$ and
$$\Delta H^{c,\,A}_s = c+\bar{V}_c(X^A_s,-A_{s-})-\bar{V}_c(X^A_{s},A_{s-}).$$
If $X^A_s\in D_{A_{s-}}$, then $\Delta H^{c,\,A}_s=0$ by construction of $\bar{V}_c$ (see~(\ref{eq:switching})). If $X^A_s\in C_{A_{s-}}$, then $\Delta H^{c,\,A}_s>0$ by Lemma~\ref{lem:ineq:Vbar}. For the control $A^c$ satisfying~(\ref{eq:def:candidate:Ac}), $A^c_{s-}\neq A_s^c$ if and only if $X^{A^c}_s\in D_{A^c_{s-}}$, and this completes the proof.
\end{proof}

We can now prove the following proposition.

\begin{Prop}\label{prop:martingale}
Let $A^c$ denote the candidate strategy satisfying~(\ref{eq:def:candidate:Ac}), let $H^c$ be defined by~(\ref{eq:def:H}) and let $c^*(\mu)$ be given by~(\ref{eq:def:c*}). If $0<c\leq c^*(\mu)$, then for all $x \in [0,1]$ and $a \in \{-1,1\}$:
\begin{enumerate}
  \item the process $(H^{c,A^c}_{t\wedge\tau^{A^c}})_{t\geq0}$ is a martingale under $\Prob_{x,a}$;
  \item for any $A\in\mathcal{A}$, the process $(H^{c,A}_{t\wedge\tau^A})_{t\geq0}$ is a submartingale under $\Prob_{x,a}$.
\end{enumerate}
\end{Prop}

In order to prove this, we first define an extension of $\bar{V}_c$. Since this function is defined from $[0,1]\times\{-1,1\}$ into $\R$, we let
\begin{equation}\label{eq:linear:envelope:V}
\bar{V}_c(x,y)=\frac{\bar{V}_c(x,1)-\bar{V}_c(x,-1)}{2}y+\frac{\bar{V}_c(x,1)+\bar{V}_c(x,-1)}{2},\quad y\in[-1,1],
\end{equation}
denote its linear interpolation on $[0,1]\times[-1,1]$, so that $\bar{V}_c(x,y)$ becomes a $C^\infty$-function in the variable $y\in[-1,1]$. Furthermore, we let $\mathcal{D}=\{a_c,b_c,1-a_c,1-b_c\}$ be the set of discontinuities of $\,\frac{\partial^2 \bar{V}_c}{\partial x^2}(x,y)$ in the variable $x$ when $0<c<c^*$.

\begin{proof}[Proof of Proposition~\ref{prop:martingale}.] Using the extension of $\bar{V}_c$, we see that $H^c(t,x,y,n)$ is $\mathcal{C}^\infty$ in the variables $t,y$ and $n$, in the variable $x$, it is $\mathcal{C}^1$ on $[0,1]$ and $\mathcal{C}^\infty$ on $[0,1]\setminus\mathcal{D}$. %Applying It\^o's Formula for processes with jumps (see e.g.~\cite[p.81]{Protter05}) and a \emph{local time-space} formula for the variable $x$ (see \cite[Theorem 3.2]{Peskir07}), we get for all $t\geq0$,
Applying the \emph{local time-space} formula \cite[Theorem 3.2]{Peskir07}, we get for all $t\geq0$,
\begin{align}
H^{c,A}_{t\wedge\tau^A}=&\,H^{c,A}_0 + \int_0^{t\wedge\tau^A}\hspace{-4pt}1\cdot\,ds + \int_0^{t\wedge\tau^A}\,\frac{\partial \bar{V}_c}{\partial x}(X^A_s, A_s)\,dX^A_s +\int_0^{t\wedge\tau^A}\,\frac{\partial \bar{V}_c}{\partial y}(X^A_s, A_{s-})\,dA_s\nonumber\\
&+\int_0^{t\wedge\tau^A}c\,dN_s(A)+ \frac{1}{2}\int_0^{t\wedge\tau^A}\,\frac{\partial^2 \bar{V}_c}{\partial x^2}(X^A_s,A_s)\ind_{\{X^A_s\notin\mathcal{D}\}}(s)\,ds\nonumber \\
&+\frac{1}{2}\int_0^{t\wedge\tau^A}\left(\,\frac{\partial \bar{V}_c}{\partial x}(X^A_s+,A_{s-})-\,\frac{\partial \bar{V}_c}{\partial x}(X^A_s-,A_{s-})\right)\ind_{\{X^A_s\in\mathcal{D}\}}(s)\,d\ell_s^{\mathcal{D}}(X^A)\nonumber\\
&+\sum_{0<s\leq {t\wedge\tau^A}}\left( \Delta H^{c,\,A}_s -c\Delta N_s(A)-\,\frac{\partial \bar{V}_c}{\partial y}(X^A_s, A_{s-})\Delta A_s\right),\label{eq:ito:formula}
\end{align}
where $\ell_s^{\mathcal{D}}(X^A)$ is the local time in $\mathcal{D}$ of the process $X^A$. By the smooth fit conditions, we have $\,\frac{\partial \bar{V}_c}{\partial x}(x+,a)-\,\frac{\partial \bar{V}_c}{\partial x}(x-,a)=0$ for all $x\in[0,1]$ and for all $a\in\{-1,1\}$. Thus, the integral with respect to the local time vanishes. If $\lambda$ stands for the Lebesgue measure, then with probability 1,
$$\ind_{\{X^A_s\notin\mathcal{D}\}}=1, \quad \mbox{ for  }\lambda-\mbox{almost all } s\in \R_+,$$
because $X^A$ is a diffusion and $\mathcal{D}$ is a finite set. Moreover, the semimartingales $(N_t(A))$ and $(A_t)$ are piecewise constant, so~(\ref{eq:ito:formula}) reduces to
$$H^{c,A}_{t\wedge\tau^A}=H^{c,A}_0   + \int_0^{t\wedge\tau^A}(1+{\mathbb L}_{A_s}\bar{V}_c(X^A_s, A_s))\,ds +\int_0^{t\wedge\tau^A}\,\frac{\partial \bar{V}_c}{\partial x}(X^A_s, A_s)\,dB_s+\sum_{0<s\leq {t\wedge\tau^A}} \Delta H^{c,\,A}_s$$
where ${\mathbb L}_a$ is the operator defined in~(\ref{eq:def:generator}) (if we need a value for $\frac{\partial^2\bar{V}_c}{\partial x^2}(x,a)$ for $x\in\mathcal{D}$, we arbitrarily take the second right derivative).

On one hand, if $A=A^c$, then by Corollary~\ref{cor:DeltaH}, $\Delta H^{c,\,A^c}_s=0$ for all $s\geq 0$. Moreover, $\{X^{A^c}_s\in C_1,A^c_s=1\}=\{A^c_s=1\}$ and $\{X^{A^c}_s\in C_{-1},A^c_s=-1\}=\{A^c_s=-1\}$. Therefore,
\begin{align*}
H^{c,A^c}_{t\wedge\tau^{A^c}}=& \,H^{c,A^c}_0 + \int_0^{t\wedge\tau^{A^c}}(1+{\mathbb L}_1\bar{V}_c(X^{A^c}_s, 1))\mathbbm{1}_{\{X^{A^c}_s\in C_1, A^c_s=1\}}\,ds \\
&+ \!\int_0^{t\wedge\tau^{A^c}}\hspace{-3pt}(1+{\mathbb L}_{-1}\bar{V}_c(X^{A^c}_s, -1))\mathbbm{1}_{\{X^{A^c}_s\in C_{-1}, A^c_s=-1\}}\,ds
+\int_0^{t\wedge\tau^{A^c}}\hspace{-3pt}\,\frac{\partial \bar{V}_c}{\partial x}(X^{A^c}_s, A^c_s)\,dB_s\\
=& \,H^{c,A^c}_0   + \int_0^{t\wedge\tau^{A^c}}\,\frac{\partial \bar{V}_c}{\partial x}(X^{A^c}_s, A^c_s)\,dB_s,
\end{align*}
by Lemma~\ref{lem:generator:Vbar}. Finally, $x\mapsto\,\frac{\partial \bar{V}_c}{\partial x}(x, \pm 1)$ is bounded on $[0,1]$, so the stochastic integral above is a martingale, which establishes the first statement.

On the other hand, for an arbitrary control $A\in\mathcal{A}$, we have by the above that for all $u<t$,
\begin{align}
H^{c,A}_{t\wedge\tau^A}-H^{c,A}_{u\wedge\tau^A}
=&\int_{u\wedge\tau^A}^{t\wedge\tau^A}\left(1+{\mathbb L}_{A_s}\bar{V}_c(X^A_s, A_s)\right)\mathbbm{1}_{\{X^A_s\in C_{A_s}\}\cup\{X^A_s\in D_{A_s}\}}\,ds \label{eq:equality:itoformula:HcA}\\
&+\int_{u\wedge\tau^A}^{t\wedge\tau^A}\,\frac{\partial \bar{V}_c}{\partial x}(X^A_s, A_s)\,dB_s+\sum_{u\wedge\tau^A< s\leq t\wedge\tau^A} \Delta H^{c,\,A}_s\nonumber\\
\geq& \int_{u\wedge\tau^A}^{t\wedge\tau^A}\,\frac{\partial \bar{V}_c}{\partial x}(X^A_s, A_s)\,dB_s +\sum_{u\wedge\tau^A< s\leq t\wedge\tau^A} \Delta H^{c,\,A}_s\label{eq:inequality:itoformula:HcA}\\
\geq& \int_{u\wedge\tau^A}^{t\wedge\tau^A}\,\frac{\partial \bar{V}_c}{\partial x}(X^A_s, A_s)\,dB_s, \nonumber
\end{align}
where we used Lemma~\ref{lem:generator:Vbar} and Corollary~\ref{cor:DeltaH}. This shows that $H^{c,A}$ is a submartingale.
\end{proof}

We are now ready to prove the optimality of our candidate strategy.

\begin{proof}[Proof of Theorem~\ref{verification:thm:min}]
Let $0<c\leq c^*(\mu)$. On one hand, by Corollary~\ref{cor:DeltaH} and by the first statement of Proposition~\ref{prop:martingale}, we have that under $\Prob_{x,a}$,
\begin{align*}
\bar{V}_c(x, a)&=H^{c,A^c}_{0-}=H^{c,A^c}_0 =\E_{x,a}\left(H^{c,A^c}_{t\wedge\tau^{A^c}}\right)\\
&=\E_{x,a}\left(t\wedge\tau^{A^c}+  cN_{t\wedge\tau^{A^c}}(A^c)+ \bar{V}_c\left(X^{A^c}_{t\wedge\tau^{A^c}}, A^c_{t\wedge\tau^{A^c}}\right)\right),
\end{align*}
for all $t\geq0$. Since $\bar{V}_c$ is a continuous and bounded function, since $N(A^c)$ is an increasing process and since by Proposition~\ref{prop:esperance:tauAc}, $\E_{x,a}\left(\tau^{A^c}\right)<+\infty$ and $\E_{x,a}\left(N_{\tau^{A^c}}(A^c)\right)<+\infty$, we get by dominated and monotone convergence that
\begin{align}
\bar{V}_c(x, a)&=\lim_{t\rightarrow\infty}\E_{x,a}\left(t\wedge\tau^{A^c}+  cN_{t\wedge\tau^{A^c}}(A^c)+ \bar{V}_c\left(X^{A^c}_{t\wedge\tau^{A^c}}, A^c_{t\wedge\tau^{A^c}}\right)\right)\nonumber\\
 &=\E_{x,a}\left(\tau^{A^c}+  cN_{\tau^{A^c}}(A^c)\right),\label{egalite_vbar}
\end{align}
as $X^{A^c}_{\tau^{A^c}}\in\{0,1\}$ and $\bar{V}_c(0, \pm1)=\bar{V}_c(1, \pm1)=0$. On the other hand, let $A\in\mathcal{A}$ be such that $J_c(x,a,A)<+\infty$. Then, by Corollary~\ref{cor:DeltaH} and by the second statement of Proposition~\ref{prop:martingale},
$$
\bar{V}_c(x, a)=H^{c,A}_{0-} \leq H^{c,A}_0 \leq \E_{x,a}\left(H^{c,A}_{t\wedge \tau^A}\right),%\qquad \Prob_{x,a}\text{-a.s.}
$$
for all $t\geq0$. As just above, we get by dominated and monotone convergence that
\begin{equation}
\bar{V}_c(x, a)\leq\lim_{t\rightarrow\infty}\E_{x,a}\left(H^{c,A}_{t\wedge\tau^A}\right)=\E_{x,a}\left(\tau^A+ cN_{\tau^A}(A)\right)\label{inegalite_vbar}.
\end{equation}
If $J_c(x,a,A)=+\infty$, it is then clear that $\bar{V}_c(x, a)<J_c(x,a,A)$. Combining~(\ref{egalite_vbar}) and~(\ref{inegalite_vbar}), we obtain
$$\bar{V}_c(x, a)=\inf_{A\in\mathcal{A}} \E_{x,a}(\tau^A+  cN_{\tau^A}(A))=V_c(x,a),$$
where the infimum is attained by the control $A^c$. This proves the optimality of the strategy $A^c$ in the case where $0<c\leq c^*(\mu)$.

%c>c*
Let $c^*(\mu)\leq c$. Then
$$V_{c^*}(x,a)=\inf_{A\in\mathcal{A}}\E_{x,a}\left(\tau^A+c^*N_{\tau^A}(A)\right)\leq \inf_{A\in\mathcal{A}}\E_{x,a}\left(\tau^A+cN_{\tau^A}(A)\right)=V_c(x,a),$$
by definition of the value function. Moreover, if $\tilde{A}$ denotes the constant strategy, then
$$J_{c^*}(x,a,\tilde{A})=J_c(x,a,\tilde{A})=\E_{x,a}(\tau^{\tilde{A}}),$$
and, again by definition of the value function,
$V_c(x,a)\leq J_c(x,a,\tilde{A}).$
Finally, by~(\ref{eq:sol:FBP:c*:V}) and by the first part of the proof, we know that
$V_{c^*}(x,a)=\bar{V}_{c^*}(x,a)=J_{c^*}(x,a,\tilde{A}).$
Hence, $V_{c^*}(x,a)=V_c(x,a)$ and the strategy $\tilde{A}$ is optimal for all $c\geq c^*(\mu)$.

Suppose now that $c>c^*(\mu)$ and that there exists another optimal strategy $\bar{A}\in\mathcal{A}$ such that there exists $(x, a)\in[0,1]\times\{-1,1\}$ with $\Prob_{x,a}(N_{\tau^{\bar{A}}}(\bar{A})>0)>0$. Then
$$\E_{x,a}\left(\tau^{\bar{A}}+cN_{\tau^{\bar{A}}}(\bar{A})\right)>\E_{x,a}\left(\tau^{\bar{A}}+c^*N_{\tau^{\bar{A}}}(\bar{A})\right)\geq V_{c^*}(x,a)=V_c(x,a),$$
which contradicts the optimality hypothesis. This shows that if $c>c^*(\mu)$, then $\tilde{A}$ is the unique optimal strategy.
\end{proof}

\subsection{Free boundary problem for the maximization problem}
The resolution of the free boundary problem \eqref{eq:FBP:Wc}, as well as the proof of the optimality of the candidate control, are similar to what we have already done in the minimization problem; we will however highlight the places where the computations differ.

\begin{proof}[Proof of Proposition~\ref{prop:sol:FBP:max}]
Let $0<c<c^*(\mu)$. The general the solution to~(\ref{eq:continuation:max}) is the same as for~(\ref{eq:ODE:continuation}), and since~(\ref{eq:continuation:max}) is satisfied in the two intervals $[0,a^{\text{max}}_c[$ and $]b^{\text{max}}_c,1]$, there are four constants to determine, which are reduced to two by the boundary conditions~(\ref{eq:boundary:condit:max}). Then,~(\ref{eq:switching:max}) and (\ref{eq:symmetry:max}) give the form of $\bar{V}^{\text{max}}_c(x,1)$ given in~(\ref{eq:sol:FBP:max:c<c*}). The four unknowns $\delta_c, \gamma_c, a^{\text{max}}_c$ and $b^{\text{max}}_c$ have to be determined using~(\ref{eq:cond:fit1:max})--(\ref{eq:smooth:fit2:max}). Using~(\ref{eq:sol:FBP:max:c<c*}), these equations give us, after some simplifications, in the same the order:
\begin{align}
\tfrac{1-2a^{\text{max}}_c}{\mu}+ \gamma_c \left(e^{-2\mu a^{\text{max}}_c}-1\right)&=\gamma_c \left(e^{-2\mu(1- a^{\text{max}}_c)}-1\right)- c,\label{eq:resolution:max1}\\
\gamma_c \left(e^{-2\mu (1-b^{\text{max}}_c)}-1\right)- c&=\tfrac{2-2b^{\text{max}}_c}{\mu}+ \delta_c \left(e^{2\mu(1-b^{\text{max}}_c)}-1\right),\label{eq:resolution:max2}\\
2 \mu\gamma_c \left(e^{-2\mu(1- a^{\text{max}}_c)} + e^{-2\mu a^{\text{max}}_c}\right)+  \tfrac{2}{\mu}&=0,\label{eq:resolution:max3}\\
2 \mu\gamma_c e^{-2\mu(1- b^{\text{max}}_c)}+2\mu\delta_c e^{2\mu (1-b^{\text{max}}_c)} + \tfrac{2}{\mu}&=0.\label{eq:resolution:max4}
\end{align}
Observe that setting $\gamma_c e^{-2\mu}=\alpha_c$ and $1-a^{\text{max}}_c=b_c$, the equations (\ref{eq:resolution:max1}) and (\ref{eq:resolution:max3}) become identical to~(\ref{eq:resolutionb}) and (\ref{eq:resolutiond}), respectively. In the proof of Proposition~\ref{prop:sol:FBP:min}, we have already established that~(\ref{eq:resolutionb}) and (\ref{eq:resolutiond}) have a unique solution $\alpha_c$ and $0<b_c<\frac{1}{2}$. Thus, it remains only to prove the existence of $\delta_c$ and $b^{\text{max}}_c$ solution of (\ref{eq:resolution:max2}) and (\ref{eq:resolution:max4}) such that $\delta_c$ satisfies~(\ref{eq:delta}) and $1-a_c<b^{\text{max}}_c<1$.

Equation~(\ref{eq:delta}) follows directly from~(\ref{eq:resolution:max4}). Substituting~(\ref{eq:delta}) into (\ref{eq:resolution:max2}), multiplying by $\mu^2$ and rearranging terms, we get~(\ref{eq:g2}). Now set
\begin{equation}\label{eq:def:hctilde:max}
\tilde{h}_c^{\text{max}}(s)= \gamma_c\mu^2e^{-2s} +e^{-s}\left(1-2\gamma_c\mu^2\right)+ s-1+\gamma_c\mu^2+c\mu^2,
\end{equation}
so that~(\ref{eq:g2}) is equivalent to $\tilde{h}_c^{\text{max}}(2\mu(1- b^{\text{max}}_c))=0$. Observe that $\tilde{h}_c^{\text{max}}(0)=c\mu^2$ and, by~(\ref{eq:alpha}), that
\begin{align*}
\tilde{h}_c^{\text{max}}(2\mu b_c)=&\,-\frac{e^\mu}{2\cosh(2\mu b_c-\mu)} e^{-4\mu b_c}  + e^{-2\mu b_c}\left(1+\frac{e^\mu}{\cosh(2\mu b_c-\mu)}\right)+2\mu b_c-1  \\
 &\qquad\qquad -\frac{e^\mu}{2\cosh(2\mu b_c-\mu)}+c\mu^2\\
=&\,-\frac{\sinh(2\mu b_c-\mu)}{\cosh(2\mu b_c-\mu)} -\frac{\sinh(\mu)}{\cosh(2\mu b_c-\mu)}+2\mu b_c+c\mu^2\\
=&\,\mu-\frac{\sinh(\mu)}{\cosh(2\mu b_c-\mu)},
\end{align*}
where we used~(\ref{eq:t0:sinh:cosh}) with $t_c=2\mu b_c-\mu$ to get the last equality. Looking back to~(\ref{eq:htilde:tc+mu}) and (\ref{eq:htilde:tc+mu:positive}), we see that
$$\tilde{h}_c^{\text{max}}(2\mu b_c)=\mu-\frac{\sinh(\mu)}{\cosh(2\mu b_c-\mu)}=-\tilde{h}_c(2\mu b_c)<0.$$
Furthermore, from~(\ref{eq:def:hctilde:max}), we see that the derivative of $\tilde{h}_c^{\text{max}}$ vanishes twice: at $0$ and at $\log(-2\gamma_c\mu^2)$ (compare also with~(\ref{eq:def:mtilde})). By~(\ref{eq:alpha}), we have that
$$\log\left(-2\gamma_c\mu^2\right)=\log\left(-2\alpha_c e^{2\mu}\mu^2\right)=\log\left(\frac{e^\mu}{\cosh(2\mu b_c-\mu)}\right)=\mu-\log(\cosh(2\mu b_c-\mu)).$$
Since $2\mu b_c-\mu<0$ because $b_c\in[0,\frac{1}{2}[$ by Proposition~\ref{prop:sol:FBP:min}, we have that $-\log(\cosh(2\mu b_c-\mu))>2\mu b_c-\mu$ which is equivalent by the preceding to $2\mu b_c<\log(-2\gamma_c\mu^2)$. Therefore, the function $\tilde{h}_c^{\text{max}}$ is monotone decreasing on $]0,2\mu b_c[$ and vanishes only once on this interval, at a value which we denote $2\mu(1-b^{\text{max}}_c)$. We have thus established the existence of a number $b^{\text{max}}_c\in\,]1-b_c,1[$ which is the unique solution of~(\ref{eq:g2}).

%comparison of bcmax with 1-ac
Finally, we check that $b^{\text{max}}_c>1-a_c$. This is clearly equivalent to showing that $2\mu(1-b^{\text{max}}_c)<2\mu a_c$, where $2\mu a_c=s_c$ is the unique solution of $\tilde{h}_c(s)=0$ on $]0,t_c+\mu[$ (see (\ref{eq:hctildes=0})). We have just shown that $\tilde{h}_c^{\text{max}}$ is strictly decreasing on $[0,t_c+\mu]$. The function $\tilde{h}_c$ is strictly increasing on $[0,t_c+\mu]$ (see the end of the proof of Proposition~\ref{prop:sol:FBP:min}). Thus, in order to prove that $2\mu(1-b^{\text{max}}_c)<s_c$, it suffices to show that $\tilde{h}_c(s)<-\tilde{h}_c^{\text{max}}(s)$ for all $s\in\,]0,t_c+\mu[\,$. We have
\begin{align*}
\tilde{h}_c(s)+\tilde{h}_c^{\text{max}}(s)&=\mu^2\alpha_c \left(e^{2 s}-2e^s+1+e^{2\mu}+e^{2\mu}e^{-2s}-2e^{2\mu}e^{-s}\right)+e^s+e^{-s}-2\\
&=\mu^2\alpha_c (e^s+e^{-s}-2)(e^s+e^{-s+2\mu})+e^s+e^{-s}-2\\
&=\mu^2\alpha_c\left(4\sinh^2\left(\tfrac{s}{2}\right)\right) \left(2e^\mu\cosh(s-\mu)\right)+4\sinh^2\left(\tfrac{s}{2}\right),
\end{align*}
which is strictly negative on $]0,t_c+\mu[$ since by~(\ref{eq:alpha}) and the fact that $-\mu<t_c<0$,
$$2\mu^2\alpha_c e^\mu\cosh(s-\mu)+1=1-\frac{\cosh(s-\mu)}{\cosh(t_c)}\,<0,\qquad \text{ for all } s\in\,]0,t_c+\mu[\,. $$
The proof of part 2 of Proposition~\ref{prop:sol:FBP:max} is complete. We note for future reference that
\begin{equation}\label{3.53a}
   2\mu(1- b^{\text{max}}_c) < t_c + \mu.
\end{equation}

%c=c*
Let us now consider the case $c=c^*(\mu)$. By the preceding, we have immediately that $a^{\text{max}}_{c^*}=1-b_{c^*}$
and that
$\gamma_{c^*}=\alpha_{c^*}e^{2\mu}.$
It remains to see that $b^{\text{max}}_{c^*}=a^{\text{max}}_{c^*}$ or, equivalently, that $2\mu(1-a^{\text{max}}_{c^*})$ is a solution of $\tilde{h}_{c^*}^{\text{max}}(s)=0$. Using the formula for $c^*$ in~(\ref{eq:def:c*}) and for $\tilde{h}_{c^*}^{\text{max}}$ in~(\ref{eq:def:hctilde:max}), the expression of $\alpha_{c^*}$ in~(\ref{eq:alpha*}) and the formula for $a^{\text{max}}_{c^*}=1-b_{c^*}$ in~(\ref{eq:ac*=bc*}) via either~(\ref{eq_def_tc*_1}) or~(\ref{eq_def_tc*_2}),  we find that
\begin{align*}
\tilde{h}_{c^*}^{\text{max}}(2\mu(1-a^{\text{max}}_{c^*}))=&\,\tilde{h}_{c^*}^{\text{max}}(t_{c^*}+\mu)\\
=&\,\gamma_{c^*}\mu^2e^{-2(t_{c^*}+\mu)}+ e^{-(t_{c^*}+\mu)}\left(1-2\gamma_{c^*}\mu^2\right)+ t_{c^*}+\mu -1 + \gamma_{c^*}\mu^2 + c^*\mu^2 \\
=&\,\gamma_{c^*}\mu^2\left(e^{-(t_{c^*}+\mu)}-1\right)^2+e^{-(t_{c^*}+\mu)} +t_{c^*}+\mu-1+c^*\mu^2\\
=&\,-\tfrac{\mu e^{\mu}}{2\sinh(\mu)}\left(e^{-\mu}\tfrac{\mu}{\sinh(\mu)}\tfrac{1}{1-\sqrt{1-\tfrac{\mu^2}{\sinh^2(\mu)}}}  -1\right)^2 +e^{-\mu}\tfrac{\mu}{\sinh(\mu)}\tfrac{1}{1-\sqrt{1-\tfrac{\mu^2}{\sinh^2(\mu)}}}\\
&\qquad -\sqrt{1-\tfrac{\mu^2}{\sinh^2(\mu)}}+\mu-1\\
=&\,-\tfrac{\mu e^{-\mu}}{2\sinh(\mu)}\tfrac{\mu^2}{\sinh^2(\mu)}\tfrac{1+2\sqrt{1-\tfrac{\mu^2}{\sinh^2(\mu)}}+1-\tfrac{\mu^2}{\sinh^2(\mu)}}{\left(\tfrac{\mu^2}{\sinh^2(\mu)}\right)^2}+ \tfrac{\mu^2}{\sinh^2(\mu)}\tfrac{1+\sqrt{1-\tfrac{\mu^2}{\sinh^2(\mu)}}}{\tfrac{\mu^2}{\sinh^2(\mu)}}\\
&\qquad -\tfrac{\mu e^{\mu}}{2\sinh(\mu)}+\tfrac{\mu e^{-\mu}}{\sinh(\mu)} \tfrac{1+\sqrt{1-\tfrac{\mu^2}{\sinh^2(\mu)}}}{\tfrac{\mu^2}{\sinh^2(\mu)}}-\sqrt{1-\tfrac{\mu^2}{\sinh^2(\mu)}}+\mu-1\\
=&\,-\tfrac{\mu e^{\mu}}{2\sinh(\mu)}+\tfrac{\mu e^{-\mu}}{2\sinh(\mu)}+\mu\\
=&\,0.
\end{align*}
Therefore,  $a^{\text{max}}_{c^*}=b^{\text{max}}_{c^*}$ and by~(\ref{eq:delta}), the last parameter $\delta_{c^*}$ is given by
$$
\delta_{c^*}=-\tfrac{1}{\mu^2}\,e^{-2\mu(1-b^{\text{max}}_{c^*})}-\gamma_c e^{-4\mu(1-b^{\text{max}}_{c^*})}=-\tfrac{1}{\mu^2}\,e^{-(t_{c^*}+\mu)}-\gamma_{c^*} e^{-2(t_{c^*}+\mu)}=-\tfrac{e^{-\mu}}{2\mu\sinh(\mu)}=\alpha_{c^*}
$$
(for the third equality, use \eqref{eq_def_tc*_2}). Replacing the value of $\gamma_{c^*}$ and $\alpha_{c^*}$ in the general expression for $\bar{V}^{\text{max}}_{c^*}$ given in~(\ref{eq:sol:FBP:max:c<c*}), we finally obtain as in~(\ref{eq:value:barVc*}) that, for all $x\in[0,1]$,
\begin{eqnarray}
\bar{V}^{\text{max}}_{c^*}(x,1)=\frac{-x}{\mu}+\frac{1-e^{-2\mu x}}{\mu(1-e^{-2\mu})}=\bar{V}_{c^*}(x,1)
\end{eqnarray}
and $\bar{V}^{\text{max}}_{c^*}(x,-1)=\bar{V}^{\text{max}}_{c^*}(1-x,1)$ from~(\ref{eq:symmetry:max}). This completes the proof.
\end{proof}

\subsection{Proof of optimality for the maximization problem}

Now that we have the solution of the free boundary problem~(\ref{eq:FBP:Wc}), we shall prove that $\bar{V}^{\text{max}}_c$ is equal to the value function $V^{\text{max}}_c$ and that $G^c$ is an optimal control. This will be similar to the proof of optimality in the minimization problem.

\begin{proof}[Proof of Proposition~\ref{prop_esperance_tauGc}.]
The proof is the same as for Proposition~\ref{prop:esperance:tauAc}.
\end{proof}

\begin{Lem}\label{lem:generator:Wbar}
Let ${\mathbb L}_a$ denote the operator defined in~(\ref{eq:def:generator}). Then, for all $0<c<c^*(\mu)$,
\begin{align*}
1+{\mathbb L}_a \bar{V}^{\text{max}}_c(x, a)&=0,     \qquad\mbox{ for all } x \in  C_a^{\text{max}},\\
1+{\mathbb L}_a \bar{V}^{\text{max}}_c(x, a)&<0,     \qquad\mbox{ for all } x \in  D_a^{\text{max}}\setminus\partial D_a^{\text{max}}.
\end{align*}
Moreover, if $c=c^*(\mu)$, then
$$1+{\mathbb L}_a \bar{V}^{\text{max}}_{c^*}(x, a)=0,     \qquad\mbox{ for all } x \in  [0,1].$$
\end{Lem}
\begin{proof}
The proof follows the same steps as for Lemma~\ref{lem:generator:Vbar}. In particular, for $x\in\,]a^{\text{max}}_c,b^{\text{max}}_c [\,$, as in \eqref{3.38a}, $1+{\mathbb L}_1 \bar{V}^{\text{max}}_c(x, 1) = 2 \mu \,\frac{\partial \bar{V}^{\text{max}}_c}{\partial x}(x,-1)$ and $1-x\in\,]1-b^{\text{max}}_c,1-a^{\text{max}}_c [\,\subset[0,a^{\text{max}}_c]$, and so
$$\bar{V}^{\text{max}}_c(x, -1)=\bar{V}^{\text{max}}_c(1-x, 1) =-\tfrac{1-x}{\mu}+\gamma_c\left(e^{-2\mu(1-x)}-1\right).$$
Therefore,
$$
\mu\frac{\partial\bar{V}^{\text{max}}_c}{\partial x}(x, -1)=1+ 2\mu^2\gamma_c e^{-2\mu(1-x)}<1+ 2\mu^2\alpha_c e^{2\mu a^{\text{max}}_c }= 1-\tfrac{e^{-\mu}}{\cosh(t_c)}e^{-t_c+\mu }=\tanh(t_c)<0,
$$
where we have used the equalities $\gamma_c=\alpha_ce^{2\mu}<0$, $a^{\text{max}}_c=1-b_c$, (\ref{eq:alpha:bis}) and $2\mu b_c=t_c+\mu$, and the fact that $t_c<0$. This proves Lemma~\ref{lem:generator:Wbar}.
\end{proof}

\begin{Lem}\label{lem:ineq:Wbar}
For all $0<c\leq c^*(\mu)$ and any $a\in\{-1,1\}$, if $x\in C_a^{\text{max}}$, then
$$\bar{V}^{\text{max}}_c(x,a)>\bar{V}^{\text{max}}_c(x,-a)-c.$$
\end{Lem}
\begin{proof}
We compare the left- and right-hand sides of this inequality on a case by case basis as in the proof of Lemma~\ref{lem:ineq:Vbar}. In the case where $c=c^*(\mu)$, the result follows immediately from~(\ref{eq:sol:FBP:c*:W}) together with~(\ref{eq:ac*:argmax}). Assume $0<c<c^*(\mu)$ and, without loss of generality, that $a=1$.
\vskip 10pt

   {\em Case 1.} If $x\in[0, 1-b^{\text{max}}_c[$, then $1-x\in\,]b^{\text{max}}_c,1]$ and according to~(\ref{eq:symmetry:max}) and~(\ref{eq:sol:FBP:max:c<c*}),
$$
\bar{V}^{\text{max}}_c(x,-1)-\bar{V}^{\text{max}}_c(x,1)-c= e_1(x),
$$
where
\begin{equation}\label{3.55a}
 e_1(x):= \tfrac{2x}{\mu} +\delta_c\left(e^{2\mu x}-1\right) -\gamma_c\left(e^{-2\mu x}-1\right)-c.
\end{equation}
The sign of the derivative $e_1'(x) = 2\mu e^{-2\mu x}(\delta_c e^{4\mu x} + \mu^{-2} e^{2\mu x} + \gamma_c)$ is determined by a quadratic function of $e^{2\mu x}$, which vanishes at most twice on $\R$, at $y_1^{\pm}$, which are given by the formula~(\ref{eq:def:roots:x1pm}) for $x_1^{\pm}$, but with $\alpha_c$ replaced by $\delta_c$ and $\beta_c$ replaced by $\gamma_c$. When $\delta_c \neq 0$, the discriminant $\Delta_1=\frac{1}{4\mu^4\delta_c^2}-\frac{\gamma_c}{\delta_c}$ is positive since
$$\Delta_1\geq0 \;\Leftrightarrow\; 1-4\mu^4\delta_c\gamma_c\geq0 \;\Leftrightarrow\; \left(1+2\mu^2\gamma_c e^{-2\mu(1-b^{\text{max}}_c)}\right)^2\geq0,
$$
by~(\ref{eq:delta}). Using again \eqref{eq:delta}, we see that $e_1'(1-b^{\text{max}}_c)=0$ and so $1-b^{\text{max}}_c\in\{y_1^\pm\}$.
If $\delta_c>0$, then $\sqrt{\Delta_1}>\frac{1}{2\mu^2\delta_c}$ (recall that $\gamma_c <0$) and $y_1^+$ is the only root of $e_1'$. The point $1-b^{\text{max}}_c=y_1^+$ is thus the global minimum of the function $e_1$ and since $e_1(0)=-c<0$, we have $e_1(x)<0$ for all $x\in[0,1-b^{\text{max}}_c[\,$.
If $\delta_c<0$, then $e_1'$ has two distinct real roots $y_1^- < y_1^+$ such that $2\mu y_1^-<\log\left(-\frac{1}{2\mu^2\delta_c}\right)<2\mu y_1^+$.  Using \eqref{eq:delta} and \eqref{eq:alpha:bis}, we see that
$$
   1-b^{\text{max}}_c=y_1^-  \quad\Leftrightarrow\quad  e^{-2\mu(1-b^{\text{max}}_c)}>-2\mu^2\delta_c \quad\Leftrightarrow\quad \cosh(t_c)<e^{\mu-2\mu(1-b^{\text{max}}_c)}.
$$
The last inequality is satisfied since $\cosh(t_c) < e^{-t_c}$ because $t_c <0$, and $2\mu(1-b^{\text{max}}_c)\in\,]0,\mu+t_c[$ by \eqref{3.53a}. Thus, $e_1$ is strictly decreasing on $[0,1-b^{\text{max}}_c[$ with $e_1(0)=-c<0$ and so $e_1(x)<0$ for all $x\in[0,1-b^{\text{max}}_c[\,$. Finally, if $\delta_c=0$, then $y_1^+=y_1^-=1-b^{\text{max}}_c$ and $e_1'(x) <0$ if and only if $x \in\,]0,1-b^{\text{max}}_c[$, and so $e_1(x)<0$ for all $x\in[0,1-b^{\text{max}}_c[\,$.
Therefore, for all possible values of $\delta_c$, we have shown that $e_1(x) <0$ for all $x \in [0,b^{\text{max}}_c[$.
\vskip 10pt

   {\em Case 2.} If $x\in[1-b^{\text{max}}_c, 1-a^{\text{max}}_c]$, then $1-x\in\,[a^{\text{max}}_c,b^{\text{max}}_c]=D_1^{\text{max}}$ and by~(\ref{eq:symmetry:max}) and~(\ref{eq:switching:max}),
$$\bar{V}^{\text{max}}_c(x,-1)-\bar{V}^{\text{max}}_c(x,1)-c= \bar{V}^{\text{max}}_c(1-x,1)-\bar{V}^{\text{max}}_c(x,1)-c=-2 c<0.$$
\vskip 10pt

   {\em Case 3.} If $x\in\,]1-a^{\text{max}}_c, a^{\text{max}}_c[$, then $1-x\in\,]1-a^{\text{max}}_c, a^{\text{max}}_c[$ and by \eqref{eq:sol:FBP:max:c<c*},
$$
\bar{V}^{\text{max}}_c(x,-1)-\bar{V}^{\text{max}}_c(x,1)-c= -c +\tfrac{2x-1}{\mu}+ \gamma_c\left(e^{2\mu (x-1)}-e^{-2\mu x}\right)=-2c+d_2(x),
$$
where $d_2(x)$ is given by~(\ref{eq:d2}). Since $a^{\text{max}}_c=1-b_c$ and since we have seen that $d_2(\cdot)$ is strictly increasing on $]b_c,1-b_c[$ with $d_2(b_c)=0$, and $d_2(1-b_c)=2c$ by \eqref{e3.41a}, we have that $-2c+d_2(x)<0$ for all $x\in\,]1-a^{\text{max}}_c, a^{\text{max}}_c[\,$.
\vskip 10pt

   {\em Case 4.} If $x\in\,]b^{\text{max}}_c,1]$, then $1-x\in [0,1-b^{\text{max}}_c[$ and by \eqref{eq:sol:FBP:max:c<c*},
$$
\bar{V}^{\text{max}}_c(x,-1)-\bar{V}^{\text{max}}_c(x,1)-c=-e_1(1-x)-2c,
$$
where $e_1$ is defined in \eqref{3.55a}. Notice that
$$
-e_1(1-b^{\text{max}}_c)-2c = -\left(\bar{V}^{\text{max}}_c(b^{\text{max}}_c,1)-\bar{V}^{\text{max}}_c(1-b^{\text{max}}_c,1)-c\right) -2c=0,
$$
by \eqref{eq:switching:max}. It follows from the properties of $e_1$ that we already discussed in part 1 that $-e_1(1-x)-2c<0$ for all $x\in\,]b^{\text{max}}_c,1]$. This proves the desired inequality in this last case and completes the proof.
\end{proof}

Let us define $K^c: \R_+ \times [0,1] \times\{-1,1\}\times\N \longrightarrow \R_+$ by
\begin{equation}\label{eq:def:K}
(t, x, a, n) \mapsto K^c(t, x, a, n):=t - cn + \bar{V}^{\text{max}}_c(x, a).
\end{equation}
For $A\in\mathcal{A}$, we consider the process $K^{c,A}_t=K^c(t,X^A_t, A_t, N_t(A))$. The next corollary follows immediately from Lemma~\ref{lem:ineq:Wbar}.

\begin{Cor}\label{cor:DeltaK}
Let $0<c\leq c^*(\mu)$, let $G^c$ denote the control satisfying~(\ref{eq:def:candidate:Gc}) and let $A\in\mathcal{A}$ be any admissible control. Then, for all $x\in[0,1]$ and $a\in\{-1,1\}$, $\Prob_{x,a}$-a.s., for all $s\in\R_+$:
(1) $\Delta K^{c,G^c}_s=0$;
  (2) $\Delta K^{c,\,A}_s\leq0$;
  (3) $\Delta K^{c,\,A}_s<0$  if and only if $X^A_s\in C^{\text{max}}_{A_{s-}}$ and $A_s\neq A_{s-}$.
\end{Cor}
This leads to the following property of the process $K^{c,\,A}$.

\begin{Prop}\label{prop:martingale:max}
Let $G^c$ denote the candidate strategy satisfying~(\ref{eq:def:candidate:Gc}), let $K^c$ be defined by~(\ref{eq:def:K}) and let $c^*(\mu)$ be given by~(\ref{eq:def:c*}). If $0<c\leq c^*(\mu)$, then:
\begin{enumerate}
  \item the process $(K^{c,G^c}_{t\wedge\tau^{G^c}})_{t\geq0}$ is a martingale under $\Prob_{x,a}$;
  \item for any $A\in\mathcal{A}$, the process $(K^{c,A}_{t\wedge\tau^A})_{t\geq0}$ is a supermartingale under $\Prob_{x,a}$.
\end{enumerate}
\end{Prop}
\begin{proof} This is established by using the local time-space formula, as well as Lemma~\ref{lem:generator:Wbar} and Corollary~\ref{cor:DeltaK}, in the same way as in the proof of Proposition~\ref{prop:martingale}.
\end{proof}

\begin{proof}[Proof of Theorem~\ref{verification:thm:max}.]
The proof is exactly the same as the one of Theorem~\ref{verification:thm:min}, using in this case the supermartingale property given in Proposition~\ref{prop:martingale:max}.
\end{proof}

\section{Further results}\label{sec:further}
 In this last section, we present a result on generic uniqueness of the optimal strategy, as well as a scaling property. We also consider the limiting case where $c\downarrow 0$ and verify that it is consistent with the classical result of~\cite{FlemingSoner06}. These statements are given only for the minimization problem but it is not difficult to see that they are also valid for the maximization problem. At the end of this section, we provide illustrations of the value functions of both problems.

 In the minimization problem, strictly speaking, we do not have uniqueness of the optimal control in general, since for $c=c^*(\mu)$, the strategy $A^{c^*}$ is equivalent to the constant strategy $\tilde{A}$ and both are optimal. It turns out, however, that this is the only case where there are two distinct optimal strategies. The case where $c > c^*(\mu)$ has already been discussed in Theorem \ref{verification:thm:min}, so we now consider the case $c \leq c^*(\mu)$.
 %Nevertheless, we will establish the following result.
\begin{Prop}[Generic uniqueness]\label{prop:uniqueness}
Let $0<c\leq c^*(\mu)$ and let $A\in\mathcal{A}$ be a strategy such that for some $(x,a)\in[0,1]\times\{-1,1\}$, either $p_1 >0$ or $0<c< c^*(\mu)$ and $p_2 >0$, where
\begin{align*}
 p_1&:= \Prob_{x,a}\left(\exists \,t\in\left[0,\tau^A\right]: X^A_t\in C_{A_{t-}} \mbox{ and } A_t\neq A_{t-}\right),\\
   p_2&:= \Prob_{x,a}\left(\exists \, t \in [0, \tau^A[: X^A_t\in D_{A_t}\right).
\end{align*}
Then $A$ is $\Prob_{x,a}$-sub-optimal, in the sense that
$V_c(x,a)<\E_{x,a}\left(\tau^A+cN_{\tau^A}(A)\right).$
\end{Prop}
\begin{Rem}
The condition $p_1>0$ means that with positive probability, the strategy $A$ prescribes at least once to switch drifts in the continuation area of the control $A^c$. The condition $p_2>0$ means that with positive probability, the strategy $A$ prescribes at least once to continue without switching in a switching region of $A^c$.
\end{Rem}
\begin{proof}
If $\E_{x,a}\left(\tau^A\right)=+\infty$, then it is clear that $\E_{x,a}\left(\tau^A+cN_{\tau^A}(A)\right)>V_{c^*}(x,a)$. Thus, we can assume that $\E_{x,a}(\tau^A)<+\infty$. Suppose first that $p_1>0$. Then by Corollary~\ref{cor:DeltaH},
$$\E_{x,a}\left(\sum_{0< s\leq \tau^A} \Delta H^{c,\,A}_s\right)>0.$$
Setting $u=0$ in~(\ref{eq:inequality:itoformula:HcA}), we find that
$$H^{c,A}_{t\wedge\tau^A}\geq H^{c,A}_0 +\int_0^{t\wedge\tau^A}\frac{\partial V_c}{\partial x} (X^A_s, A_s)\,dB_s+\sum_{0< s\leq {t\wedge\tau^A}} \Delta H^{c,\,A}_s,$$
and taking the expectations, applying the monotone convergence and the dominated convergence theorems (recall that $V$ is bounded), we get
$$\E_{x,a}\left(H^{c,A}_{\tau^A}\right)\geq H^{c,A}_0+\E_{x,a}\left(\sum_{0< s\leq \tau^A} \Delta H^{c,\,A}_s\right)>H^{c,A}_0=V_c(x,a).$$
Since the left-hand side is equal to $\E_{x,a}\left(\tau^A+cN_{\tau^A}(A)\right)$, this proves the statement when $p_1>0$.

Suppose now that $p_2>0$. Let $\lambda$ denote Lebesgue measure. Notice that the interior int$(D_{A_t})$ of $D_{A_t}$ is not empty for $0<c<c^*(\mu)$. By right-continuity of $s \mapsto A_s$ and because of the irregular behavior of sample paths of diffusion processes, on the event $\{\exists \, t \in [0, \tau^A[: X^A_t\in D_{A_t}\}$,
\begin{equation}\label{e4.1}
   \lambda\{s \in [0,\tau^A[: X^A_s \in \mbox{int}(D_{A_s})\} >0.
\end{equation}
Setting $u=0$ in~(\ref{eq:equality:itoformula:HcA}) and then applying successively Lemma~\ref{lem:generator:Vbar} and Corollary~\ref{cor:DeltaH}, we get
$$H^{c,A}_{t\wedge\tau^A}\geq H^{c,A}_0 +\int_0^{t\wedge\tau^{A}}\frac{\partial V_c}{\partial x}(X^A_s, A_s)\,dB_s+ \int_0^{t\wedge\tau^{A}}\left(1+{\mathbb L}_{A_s}V_c(X^A_s, A_s)\right)\mathbbm{1}_{\{X^A_s\in D_{A_s}\}}\,ds.$$
Again by Lemma~\ref{lem:generator:Vbar}, the integrand of the last integral is strictly positive if $X^A_s$ belongs to the interior of $D_{A_s}$. Taking expectations in the previous inequality and applying successively the dominated and monotone convergence theorems and the hypothesis $p_2>0$ with\eqref{e4.1}, we find that
$$\E_{x,a}\left(\tau^A+cN_{\tau^A}(A)\right)=\E_{x,a}\left(H^{c,A}_{\tau^A}\right)>H^{c,A}_0=V_c(x,a),$$
which concludes the proof.
\end{proof}

If we consider a diffusion coefficient $\sigma\neq1$ for the particle, we can deduce from Theorem~\ref{verification:thm:min} the corresponding value function and  optimal control.  Let $\sigma>0$ and let $(\hat{B}_t)$ be a standard Brownian motion, $(\hat{\Filt}_t)$ be its natural filtration and $\hat{\mathcal{A}}$ be the associated set of strategies. For $\hat{A}\in\hat{\mathcal{A}}$, consider the s.d.e.
\begin{equation}\label{eq:SDE:state:scaling}
d\hat{X}^{\hat{A}}_t =\hat{A}_t\mu\,dt+\sigma d\hat{B}_t,\qquad \hat{X}^{\hat{A}}_0=x,
\end{equation}
and the corresponding stochastic control problem whose value function is given by
\begin{equation}\label{eq:control:problem:scaling}
\hat{V}(x,a,c,\mu,\sigma)=\inf_{\hat{A}\in\hat{\mathcal{A}}} \E_{x,a}(\hat{\tau}^{\hat{A}} + c \hat{N}_{\tau^{\hat{A}}}(\hat{A})),
\end{equation}
where $\hat{\tau}^{\hat{A}}$ is the exit time of $\hat{X}^{\hat{A}}$ from $[0,1]$ and $\hat{N}(\hat{A})$ is the process that counts the discontinuities of $\hat{A}$.

\begin{Prop}[Scaling property]\label{prop:scaling}
The optimal control of problem~(\ref{eq:control:problem:scaling}) is obtained by the construction that leads to~(\ref{eq:def:candidate:Ac}), but replacing $c$ by $\sigma^2 c$ and $\mu$ by $\mu/\sigma^2$. The value function satisfies
$$\hat{V}(x,a,c,\mu,\sigma)=\frac{1}{\sigma^2}\hat{V}\left(x,a,c\sigma^2,\frac{\mu}{\sigma^2},1\right),$$
where $\hat{V}\left(x,a,c\sigma^2,\mu/\sigma^2,1\right)$ coincides with the value function of problem~(\ref{eq:control:problem:min}) with $c$  replaced by $c\sigma^2 $ and $\mu$ replaced by $\mu/\sigma^2$.
\end{Prop}
\begin{proof}
Define $B_t=\sigma \hat{B}_{t/\sigma^2}$, so that $(B_t)$ is a standard Brownian motion. Setting
$$d\hat{Z}_t=\pm\mu\,dt+\sigma d\hat{B}_t, \qquad dZ_t=\pm\frac{\mu}{\sigma^2}\,dt+dB_t,$$
and $\hat{Z}_0=Z_0$, we see that $\hat{Z}_t=Z_{\sigma^2t}$.

Let $(\Filt_t)$ be the natural filtration of $(B_t)$, that is, $\Filt_t=\hat{\Filt}_{t/\sigma^2}$, and let $\mathcal{A}$ be the set of strategies associated with $(\Filt_t)$. Given $\hat{A}\in\hat{\mathcal{A}}$, define $A=(A_t)$ by $A_t=\hat{A}_{t/\sigma^2}$. This defines a one-to-one correspondence between $\hat{\mathcal{A}}$ and $\mathcal{A}$.

Let
$dX^A_t=A_t\mu\sigma^{-2}\,dt+dB_t$, with $X^A_0=x.$
Then, $X^A_{\sigma^2t}=\hat{X}^{\hat{A}}_t$ and we see that
$\tau^A=\sigma^2\hat{\tau}^{\hat{A}}$  and $N_{\sigma^2t}(A)=\hat{N}_t(\hat{A}).$
Therefore,
$$\sigma^2\left(\hat{\tau}^{\hat{A}}+c\hat{N}_{\hat{\tau}^{\hat{A}}}(\hat{A})\right)=\tau^A+c\sigma^2 N_{\tau^A}(A).$$
Minimizing the right-hand side is precisely the problem~(\ref{eq:control:problem:min}), with $c$  replaced by $c\sigma^2 $ and $\mu$ replaced by $\mu/\sigma^2$. This proves the proposition.
\end{proof}

As mentioned in the introduction, in the case where there is no switching cost ($c=0$), the solution of the control problem corresponding to~(\ref{eq:control:problem:min}) is now classical (see \cite[IV.5]{FlemingSoner06}). The value function does not depend on the initial drift and is given by
\begin{equation}\label{eq:V:limit0}
V(x)= \frac{\frac{1}{2}-\left|\frac{1}{2}-x\right|}{\mu}-\frac{1}{2\mu^2}e^{-\mu} \left(e^{2\mu\left(\frac{1}{2}-\left|\frac{1}{2}-x\right|\right)}-1\right), \qquad x\in[0,1].
\end{equation}
The optimal control, which would not be admissible in our setting, is given by
\begin{equation}
A_t =\sgn\left(\tfrac{1}{2}-X^A_t\right), %\label{eq:A:limit0}\\
\qquad \mbox{with } \qquad dX^A_t =A_t \mu\,dt+dB_t \quad\mbox{ on } \{t\leq\tau^A\}, \label{eq:eds:XA0}
\end{equation}
where $\sgn x=1$ if $x\geq0$ and $\sgn x=-1$ if $x<0$.  We observe that this control $A$ is not piecewise constant, since it corresponds to switching regions given by $D_1=[0,\frac{1}{2}]$ and $D_{-1}=[\frac{1}{2},0]$. Thus, there exists only a weak solution of~(\ref{eq:eds:XA0}) which is given by Tanaka's formula (see e.g.~\cite[Sections 7.3 and 10.4]{Chung_williams}). In the next proposition, we show that \eqref{eq:V:limit0} can be obtained as a limit of $V_c(x,a)$ as $c\downarrow0$.

\begin{Prop}\label{prop:limit:cto0}
Let $V(x)$ denote the function defined in~(\ref{eq:V:limit0}). The solution $\{\bar{V}_c,a_c,b_c\}$ of~(\ref{eq:FBP:Vc}) satisfies
%\begin{enumerate}
%   \item
$\displaystyle\lim_{c\downarrow 0} a_c=0$,  $\displaystyle\lim_{c\downarrow 0} b_c=\tfrac{1}{2}\,$ and
%   \item
$\displaystyle\lim_{c\downarrow 0}\bar{V}_c(x,a)=V(x)$, for all $x\in[0,1]$ and $a\in\{\pm1\}$.
% \end{enumerate}
\end{Prop}
\begin{proof}
Observe first that when $c\downarrow 0$, equations~(\ref{eq:bc}) and~(\ref{eq:ac}) become respectively
\begin{align*}
e^{4\mu x-2\mu} ( 2\mu x-\mu-1) + 2\mu x-\mu+1&=0,\\
\mu^2\alpha_0 e^{4\mu y}+ (1-2\mu^2\alpha_0)e^{2\mu y} +\mu^2\alpha_0 -2\mu y-1&=0,
\end{align*}
where $\alpha_0=-\frac{1}{2\mu^2}e^{-\mu}$. The unique solution of the first equation is $b_0=\frac{1}{2}$, so $\lim_{c\downarrow 0} b_c = \frac{1}{2}$. The second equation has exactly two solutions, one of which is $> \frac{1}{2}$ and the other is $a_0 = 0$. Since $ 0<a_c<\frac{1}{2}$, we deduce that $\lim_{c\downarrow 0} a_c = 0$. Putting these values into~(\ref{eq:alpha}) and~(\ref{eq:beta}), respectively, we find that when $c\downarrow0$, $\alpha_c\rightarrow-\frac{1}{2\mu^2}e^{-\mu}$ and $\beta_c\rightarrow\frac{e^{-\mu}}{2\mu^2}-\frac{1}{\mu^2}\,$. Therefore, according to~(\ref{eq:sol:FBP:min:c<c*}), we find that
$$\lim_{c\downarrow 0} V_c(x,a) =\left\{\begin{array}{ll}
                    %\displaystyle
                    \frac{x}{\mu}-\frac{1}{2\mu^2}e^{-\mu} \left(e^{2\mu x}-1\right), & x\in\left[0,\frac{1}{2}\right], \\[7pt]
                    %\displaystyle
                    \frac{1-x}{\mu}-\frac{1}{2\mu^2}e^{-\mu} \left(e^{2\mu(1-x)}-1\right), & x\in\,\left]\frac{1}{2},1\right]. \\
                    \end{array}\right.$$
These formulas coincide with~(\ref{eq:V:limit0}) and the proof is complete.
\end{proof}

In Figure~\ref{figure:comparison:WVf}, we give the graph of the value function of both problems~(\ref{eq:control:problem:min}) and~(\ref{eq:control:problem:max}) for two possible values of the switching cost ($c=0.01$ and $c=0.04$) when the intensity of the drift is $\mu=1$. In this case, the critical value of the cost is $c^*(1)\approx 0.058$. The numerical value of the switching boundaries in the case where $c=0.01$ are given by:
\begin{equation*}
  a_c \approx 0.0882, \qquad
  b_c \approx 0.3426, \qquad
  a^{\text{max}}_c=1-b_c, \qquad
  b^{\text{max}}_c \approx 0.9387,
\end{equation*}
and in the case where $c=0.04$, they are given by:
\begin{equation*}
  a_c \approx 0.1737, \qquad
  b_c \approx 0.2451, \qquad
  a^{\text{max}}_c=1-b_c, \qquad
  b^{\text{max}}_c \approx 0.8494.
\end{equation*}
Because of the symmetry, Figure \ref{figure:comparison:WVf} shows only the value function corresponding to a positive initial drift.

\begin{figure}[H]
\begin{center}
\includegraphics[width=7cm]{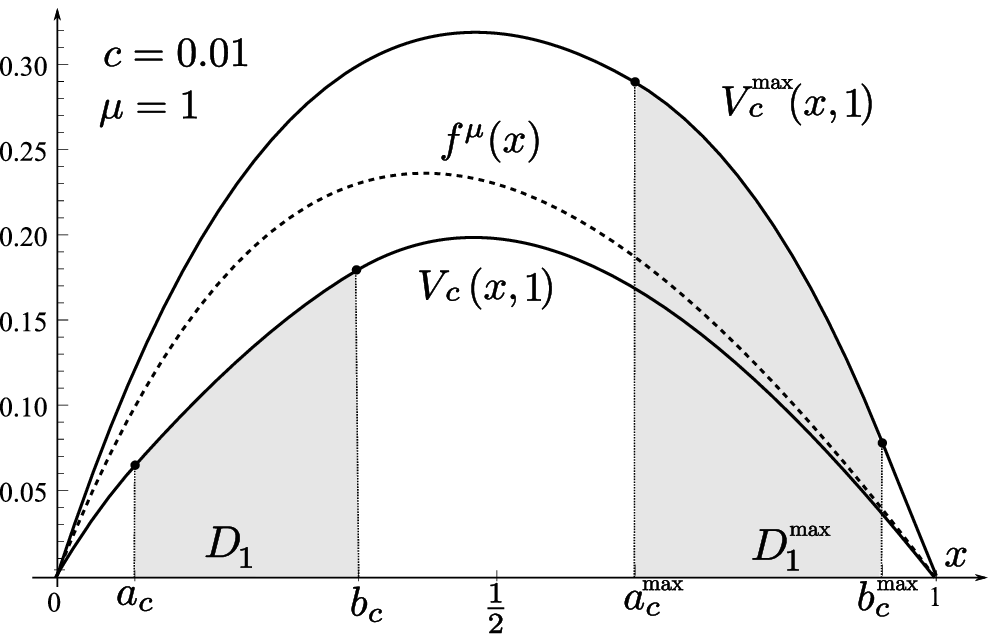}\hspace{1cm}
\includegraphics[width=7cm]{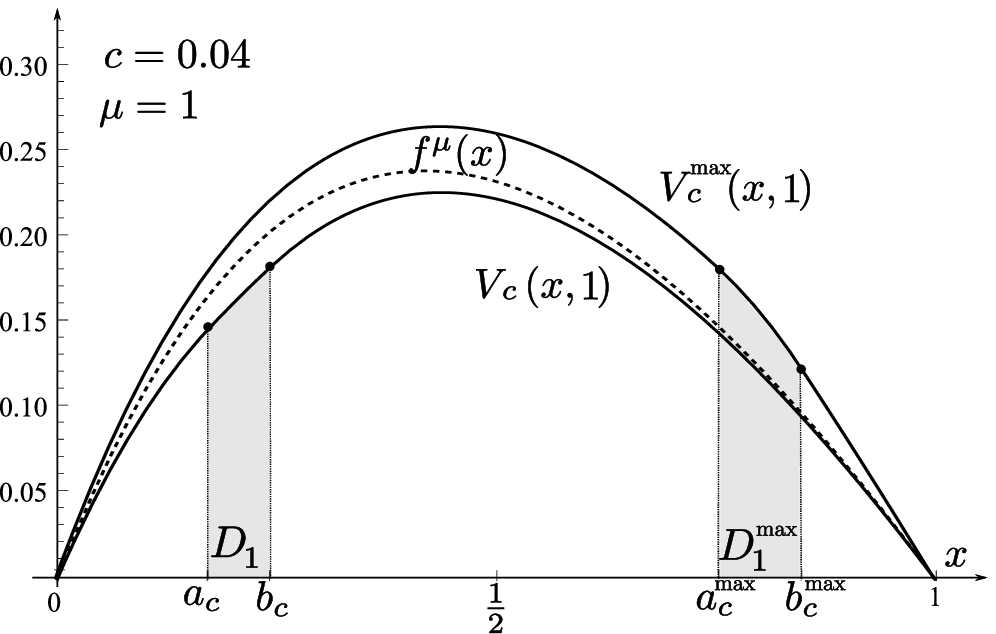}
\caption{Graphs of $V^{\text{max}}_c(x,1)$ and $V_c(x,1)$ for two different values of the cost. The dotted line is the graph of the function $f^\mu(x)$ defined in~(\ref{eq:def:fnu}), which coincides with the payoff of the constant strategy.}\label{figure:comparison:WVf}
\end{center}
\end{figure}

%\todo[inline, color=green!20]{Application?}
%\todo[inline, color=green!20]{references: ajouter shepp, shepp-shiryaev?}

\addcontentsline{toc}{section}{References}

%\bibliographystyle{acm}
%\bibliography{../../bibliography}

% Ci-dessous, je donne la bibliographie ‡ la maniËre "classique", pour que vous puissiez y rajouter des entrÈes ou la modifier.
% S'il faut changer le style, je pourrai le faire ‡ partir de ma base de donnÈes (automatiquement).

\vskip 16pt

Institut de math\'ematiques

\'Ecole Polytechnique F\'ed\'erale de Lausanne

Station 8

CH-1015 Lausanne

Switzerland
\vskip 12pt

robert.dalang@epfl.ch, laura.vinckenbosch@gmail.com

\end{document}